\documentclass[12pt]{amsart}

\usepackage{latexsym}
\usepackage{amsmath}
\usepackage{amssymb}
\usepackage{amscd}
\usepackage{amsthm}

\theoremstyle{plain}
 \newtheorem{theorem}{Theorem}[subsection]
 \newtheorem{proposition}[theorem]{Proposition}
 \newtheorem{lemma}[theorem]{Lemma}
 \newtheorem{corollary}[theorem]{Corollary}
 
\theoremstyle{definition}
 \newtheorem{definition}[theorem]{Definition}

\theoremstyle{remark}
 \newtheorem{remark}[theorem]{Remark}

\makeindex

\begin{document}
\title[Rationality of $\mathfrak{M}_3$]{The rationality of the moduli space of curves of genus $3$ after P.
Katsylo} 

\author{Christian B\"ohning}

\maketitle

\begin{abstract}
This article is a survey of P. Katsylo's proof that the moduli space $\mathfrak{M}_3$ of smooth projective complex 
 curves of genus $3$ is rational. We hope to make the argument more comprehensible and transparent by emphasizing
the underlying geometry in the proof and its key structural features. 
\end{abstract}

\tableofcontents

\section{Introduction}

The question whether $\mathfrak{M}_3$ is a rational variety or not had been open for a long time until an
affirmative answer was finally given by P. Katsylo in 1996. There is a well known transition in the behaviour of
the moduli spaces $\mathfrak{M}_g$ of smooth projective complex curves of genus $g$ from unirational for small $g$
to general type for larger values of $g$; the moral reason that $\mathfrak{M}_3$ should have a good chance to be
rational is that it is birational to a quotient of a projective space by a \emph{connected} linear algebraic group.
No variety of this form has been proved irrational up to now. More precisely, $\mathfrak{M}_3$ is birational to
the moduli space of plane quartic curves for $\mathrm{PGL}_3\,\mathbb{C}$-equivalence. All the moduli spaces
$C(d)$ of plane curves of given degree $d$ are conjectured to be rational (see \cite{Dol2}, p.162; in fact, there
it is conjectured that all the moduli spaces of hypersurfaces of given degree $d$ in $\mathbb{P}^n$ for the
$\mathrm{PGL}_{n+1}\,\mathbb{C}$-action are rational. I do not know if this conjecture should be attributed to
Dolgachev or someone else).\\
There are heuristic reasons that the spaces $C(d)$ should be rational at least for all large enough values for
$d$. Maybe it is not completely out of reach to prove this rigorously. We hope to return to this problem in the
future. In any case one might hazard the guess that irregular behaviour of $C(d)$ is most likely to be found for
small values of $d$, and showing rationality for $C(4)$ turned out to be exceptionally hard.\\
Katsylo's proof is long and computational, and, due to the importance of the result, it seems desirable to give a
more accessible and geometric treatment of the argument.\\
This paper is divided into two main sections (sections 2 and 3) which are further divided into subsections.
Section 2 treats roughly the contents of Katsylo's first paper \cite{Kat1} and section 3 deals with his second
paper \cite{Kat2}.\\
Finally I would like to thank Professor Yuri Tschinkel for proposing the project and many useful discussions.
Moreover, I am especially grateful to Professor Fedor Bogomolov with whom I discussed parts of the project and
who provided a wealth of helpful ideas.

\section{A remarkable $(\mathrm{SL}_3\,\mathbb{C} , \mathrm{SO}_3\,\mathbb{C})$-section}

\subsection{$(G,H)$-sections and covariants}
A general, i.e. nonhyperelliptic, smooth projective curve $C$ of genus $3$ is realized as a smooth plane quartic
curve via the canonical embedding, whence $\mathfrak{M}_3$ is birational to the orbit space
$C(4):=\mathbb{P}(H^0(\mathbb{P}^2, \mathcal{O}(4))/\mathrm{SL}_3\,\mathbb{C}$. We remark that whenever one has an
affine algebraic group $G$ acting on an irreducible variety $X$, then, according to a result of Rosenlicht, there
exists a nonempty invariant open subset $X_0\subset X$ such that there is a geometric quotient for the action of
$G$ on $X_0$ (cf. \cite{Po-Vi}, thm. 4.4). In the following we denote by $X/G$ any birational model of this
quotient, i.e. any model of the field $\mathbb{C}(X)^G$  of invariant rational functions.\\
The number of methods to prove rationality of quotients of projective spaces by connected reductive groups is
quite limited (cf. \cite{Dol1} for an excellent survey). The only approach which our problem is immediately
amenable to seems to be the method of $(G,H)$-sections. (There are two other points of view I know of: The first is
based on the remark that if we have a nonsingular plane quartic curve $C$, the double cover of $\mathbb{P}^2$
branched along
$C$ is a Del Pezzo surface of degree $2$, and conversely, given a Del Pezzo surface $S$ of degree $2$, then $|
-K_S |$ is a regular map which exhibits $S$ as a double cover of $\mathbb{P}^2$ branched along a plane quartic
$C$; this sets up a birational isomorphism between $\mathfrak{M}_3$ and $\mathfrak{D}\mathfrak{P}(2)$, the moduli
space of Del Pezzo surfaces of degree $2$. We can obtain such an $S$ by blowing up $7$ points in $\mathbb{P}^2$,
and one can prove that $\mathfrak{D}\mathfrak{P}(2)$ is birational to the quotient of an open subset of
$P_2^7:=(\mathbb{P}^2)^7/\mathrm{PGL}_3\,\mathbb{C}$, the configuration space of $7$ points in $\mathbb{P}^2$
(which is visibly rational), modulo an action of the Weyl group $W(E_7)$ of the root system of type $E_7$ by
Cremona transformations (note that $W(E_7)$ coincides with the permutation group of the $(-1)$-curves on $S$ that
preserves the incidence relations between them). This group is a rather large finite group, in fact, it has order
$2^{10}\cdot 3^4\cdot 5\cdot 7$. This approach does not seem to have led to anything definite in the direction of
proving rationality of $\mathfrak{M}_3$ by now, but see
\cite{D-O} for more information.\\
The second alternative, pointed out by I. Dolgachev, is to remark that $\mathfrak{M}_3$ is birational to
$\mathfrak{M}_3^{\mathrm{ev}}$, the moduli space of genus $3$ curves together with an even theta-characteristic;
this is the content of the classical theorem due to G. Scorza. The latter space is birational to the space of
nets of quadrics in $\mathbb{P}^3$ modulo the action of $\mathrm{SL}_4\,\mathbb{C}$, i.e. $\mathrm{Grass}(3,
\mathrm{Sym}^2\, (\mathbb{C}^4)^{\vee }) / \mathrm{SL}_4\,\mathbb{C}$. See \cite{Dol3}, 6.4.2, for more on this.
Compare also \cite{Kat0}, where the rationality of the related space\\ $\mathrm{Grass}(3,
\mathrm{Sym}^2\, (\mathbb{C}^5)^{\vee }) / \mathrm{SL}_5\,\mathbb{C}$ is proven; this proof, however, cannot be 
readily adapted to our situation, the difficulty seems to come down to that $4$, in contrast to $5$, is even).
\begin{definition}
Let $X$ be an irreducible variety with an action of a linear algebraic group $G$, $H< G$ a subgroup. An
irreducible subvariety $Y\subset X$ is called a \emph{$(G,H)$-section} of the action of $G$ on X if 
\begin{itemize}
\item[(1)] $\overline{G\cdot Y}=X$ ;
\item[(2)] $H\cdot Y\subset Y$;
\item[(3)] $g\in G$, $gY\cap Y \neq \emptyset$ $\implies$ $g\in H$.
\end{itemize} 
\end{definition}
In this situation $H$ is the normalizer $N_G(Y):=\{g \in G\, | \, gY\subset Y\}$ of $Y$ in $G$. The following
proposition collects some properties of $(G,H)$-sections.
\begin{proposition}
\begin{itemize}
\item[(1)] The field $\mathbb{C}(X)^G$ is isomorphic to the field $\mathbb{C}(Y)^H$ via restriction of functions
to $Y$.
\item[(2)] Let $Z$ and $X$ be $G$-varieties, $f : Z \to X$ a dominant $G$-morphism, $Y$ a $(G,H)$-section of $X$,
and $Y'$ an irreducible component of $f^{-1}(Y)$ that is $H$-invariant and dominates $Y$. Then $Y'$ is a
$(G,H)$-section of $Z$.
\end{itemize}
\end{proposition}
Part (2) of the proposition suggests that, to simplify our problem of proving rationality of $C(4)$, we should
look at covariants $\mathrm{Sym}^4\,(\mathbb{C}^3)^{\vee} \to \mathrm{Sym}^2\, (\mathbb{C}^3)^{\vee}$  of low
degree ($\mathbb{C}^3$ is the standard representation of $\mathrm{SL}_3\,\mathbb{C}$). The highest weight theory of
Cartan-Killing allows us to decompose $\mathrm{Sym}^i(\mathrm{Sym}^4\, (\mathbb{C}^3)^{\vee})$, $i\in\mathbb{N}$,
into irreducible subrepresentations (this is best done by a computer algebra system, e.g.
\ttfamily{Magma}\rmfamily) and pick the smallest $i$ such that $\mathrm{Sym}^2\,(\mathbb{C}^3)^{\vee}$ occurs as
an irreducible summand. This turns out to be $5$ and $\mathrm{Sym}^2\, (\mathbb{C}^3)^{\vee}$ occurs with
multiplicity
$2$.\\ For nonnegative integers $a,\: b$ we denote by $V(a,b)$ the irreducible $\mathrm{SL}_3\,\mathbb{C}$-module
whose highest weight has numerical labels  $a,\: b$.\\
Let us now describe the two resulting independent covariants
\begin{gather*}
\alpha_1,\:\alpha_2 : V(0,4) \to V(0,2)
\end{gather*}
of order $2$ and degree $5$ geometrically. We follow a classical geometric method of Clebsch to pass from
invariants of binary forms to contravariants of ternary forms (see \cite{G-Y}, \S 215). The covariants $\alpha_1$,
$\alpha_2$ are described in Salmon's treatise \cite{Sal}, p. 261, and p. 259, cf. also
\cite{Dix}, p. 280-282. We start by recalling the structure of the ring of $\mathrm{SL}_2\,\mathbb{C}$-invariants
of binary quartics (\cite{Muk}, section 1.3, \cite{Po-Vi}, section 0.12).
\subsection{Binary quartics}
Let 
\begin{gather}
f_4=\xi_0x_0^4+4\xi_1x_0^3x_1+6\xi_2x_0^2x_1^2+4\xi_3x_0x_1^3+\xi_4x_1^4
\end{gather}
be a general binary quartic form. The invariant algebra $R=\mathbb{C}[\xi_0,\dots
,\xi_4]^{\mathrm{SL}_2\,\mathbb{C}}$ is freely generated by two homogeneous invariants $g_2$ and $g_3$ (where
subscripts indicate degrees):
\begin{gather}
g_2(\xi)=\det\left(\begin{array}{cc} \xi_0 & \xi_2 \\ \xi_2 & \xi_4 \end{array}\right) -
4\det\left(\begin{array}{cc} \xi_1 & \xi_2 \\ \xi_2 & \xi_3 \end{array}\right)\: ,\\
g_3(\xi )=\det\left(\begin{array}{ccc} \xi_0 & \xi_1 & \xi_2 \\ \xi_1 & \xi_2 & \xi_3\\ \xi_2 & \xi_3 & \xi_4
\end{array}\right)\: .
\end{gather}
If we identify $f_4$ with its zeroes $z_1,\dots , z_4\in\mathbb{P}^1=\mathbb{C}\cup\{\infty\}$ and write 
\begin{gather*}
\lambda =\frac{(z_1-z_3)(z_2-z_4)}{(z_1-z_4)(z_2-z_3)}
\end{gather*}
for the cross-ratio, then
\begin{gather*}
g_3=0  \iff \lambda = -1,\: 2, \: \mathrm{or}\: \frac{1}{2}\: ,\\
g_2=0 \iff \lambda = -\omega \:\mathrm{or}\: -\omega^2\; \mathrm{with}\; \omega=e^{\frac{2\pi i}{3}}\: ,
\end{gather*}
the first case being commonly referred to as \emph{harmonic cross-ratio}, the second as \emph{equi-anharmonic
cross-ratio} (see \cite{Cl}, p. 171; the terminology varies a lot among different authors, however).\\
Clebsch's construction is as follows: Let $x,\: y, \: z$ be coordinates in $\mathbb{P}^2$, and let $u,\: v, \: w$
be coordinates in the dual projective plane $(\mathbb{P}^2)^{\vee}$. Let $\varphi = ax^4+4bx^3y^4+\dots  $ be a
general ternary quartic. We want to consider those lines in $\mathbb{P}^2$ such that their intersection with the
associated quartic curve $C_{\varphi}$ is a set of points whose cross-ratio is harmonic resp. equi-anharmonic. Writing
a line as $ux+vy+wz=0$ and substituting in (2) resp. (3), we see that in the equi-anharmonic case we get a quartic in
$(\mathbb{P}^2)^{\vee}$, and in the harmonic case a sextic. More precisely this gives us two
$\mathrm{SL}_3\,\mathbb{C}$-equivariant polynomial maps
\begin{gather}
\sigma  :  V(0,4) \to V(0,4)^{\vee}  \: ,\\
\psi  :  V(0,4) \to V(0,6)^{\vee} \: ,
\end{gather}
and $\sigma$ is homogeneous of degree $2$ in the coefficients of $\varphi$ whereas $\psi$ is homogeneous of
degree $3$ in the coefficients of $\varphi$ (we say $\sigma$ is a \emph{contravariant} of degree $2$ on $V(0,4)$
with values in $V(0,4)$, and analogously for $\psi$). Finally we have the \emph{Hessian covariant} of $\varphi$:
\begin{gather}
\mathrm{Hess} : V(0,4) \to V(0,6)
\end{gather}
which associates to $\varphi$ the determinant of the matrix of second partial derivatives of $\varphi$. It is of
degree $3$ in the coefficients of $\varphi$.\\
We will now cook up $\alpha_1$, $\alpha_2$ from $\varphi,\: \sigma , \: \psi ,\: \mathrm{Hess}$: Let $\varphi$
operate on $\psi$; by this we mean that if $\varphi =ax^4+4b x^3y +\dots $ then we act on $\psi$ by the
differential operator
\begin{gather*}
a\frac{\partial^4}{\partial u^4}+ 4b\frac{\partial^4}{\partial u^3 \partial v} + \dots
\end{gather*}
(i.e. we replace a coordinate by partial differentiation with respect to the dual coordinate). In this way we get
a contravariant $\rho$ of degree $4$ on $V(0,4)$ with values in $V(0,2)$. If we operate with $\rho$ on $\varphi$
we get $\alpha_1$.\\
We obtain $\alpha_2$ if we operate with $\sigma$ on $\mathrm{Hess}$.\\
This is a geometric way to describe $\alpha_1$, $\alpha_2$. For every $c=[c_1 : c_2]\in \mathbb{P}^1$ we get in
this way a rational map
\begin{gather}
f_c=c_1\alpha_1 + c_2\alpha_2: \mathbb{P}(V(0, 4)) \dasharrow \mathbb{P}(V(0,2))\: .
\end{gather}
For the special quartics
\begin{gather}
\varphi =ax^4+by^4+cz^4+6fy^2z^2+6gz^2x^2+6hx^2y^2
\end{gather}
the quantities $\alpha_1$ and $\alpha_2$ were calculated by Salmon in \cite{Sal}, p. 257 ff. We reproduce the
results here for the reader's convenience. Put
\begin{gather}
L:=abc \, ,\: P:=af^2+bg^2+ch^2\, ,\\ \: R:=fgh \: ; \nonumber
\end{gather}
Then
\begin{gather}
\alpha_1=(3L+9P+10R)(afx^2+bgy^2+chz^2)+\\
(10L+2P+4R)(ghx^2+hfy^2+fgz^2)\nonumber \\-12(a^2f^3x^2+b^2g^3y^2+c^2h^3z^2)\: ; \nonumber\\
\alpha_2=(L+3P+30R)(afx^2+bgy^2+chz^2)+\\
(10L-6P-12R)(ghx^2+hfy^2+fgz^2)\nonumber \\-4(a^2f^3x^2+b^2g^3y^2+c^2h^3z^2)\: . \nonumber
\end{gather}
Note that the covariant conic $-\frac{1}{20}(\alpha_1-3\alpha_2)$ looks a little simpler.\\
Let us see explicitly, using (8)-(11), that $f_c$ is dominant for every $c\in\mathbb{P}^1$; for $a=b=c=f=g=h=1$ we
get $\alpha_1 = 48(x^2+y^2+z^2)$, $\alpha_2=16(x^2+y^2+z^2)$, so the image of $\varphi$ under $f_c$ in this case is
a nonsingular conic unless $c=[-1 : 3]$. But for $a=1,\: b=c=0,\: f=g=h=1$ we obtain $\alpha_1=13x^2+6y^2+6z^2$,
$\alpha_2=11x^2-18y^2-18z^2$, and for these values $-\alpha_1+3\alpha_2$ defines a nonsingular conic.\\
Let $\mathcal{L}_c$ be the linear system generated by $6$ quintics which defines $f_c$ and let $B_c$ be its base
locus; thus $U_c:=\mathbb{P}(V(0,4))\backslash B$ is an $\mathrm{SL}_3\,\mathbb{C}$-invariant open set, and if
$f_{c,0}:=f_c|_{U_c}$, then $X_c:=f_{c, 0}^{-1}(\mathbb{C}h_0)$, where $h_0$ defines a non-singular conic, is a
good candidate for an $(\mathrm{SL}_3\,\mathbb{C} ,\, \mathrm{SO}_3\,\mathbb{C} )$-section of $U_c$. We choose
$h_0=xz-y^2$.\\
\begin{proposition}
$X_c$ is a smooth irreducible $\mathrm{SO}_3\,\mathbb{C}$-invariant variety, $\overline{\mathrm{SL}_3\,
\mathbb{C}\cdot X }=\mathbb{P}(V(0,4))$, and the normalizer of $X_c$ in $\mathrm{SL}_3\,\mathbb{C}$ is exactly
$\mathrm{SO}_3\,\mathbb{C}$. $X_c$ is an
$(\mathrm{SL}_3\mathbb{C} ,\, \mathrm{SO}_3\,\mathbb{C})$-section of $U_c$.
\end{proposition}
\begin{proof}
The $\mathrm{SO}_3\,\mathbb{C}$-invariance of $X_c$ follows from its construction. We show that the differential
$d (f_{c,0})_x$ is surjective for all $x\in X_c$: In fact, 
\begin{gather*}
d(f_{c,0})_x (T_x U_c)\supset d(f_{c, 0})_x (\mathfrak{sl}_3(x))=\mathfrak{sl}_3 (f_{c,0}(x))=T_{\mathbb{C}h_0}\, 
\mathbb{P} V(0, 2)
\end{gather*}
Here $\mathfrak{sl}_3(x)$ denotes the tangent space to the $\mathrm{SL}_3\,\mathbb{C}$-orbit of $x$ in $U_c$, i.e.
if $O_x \, :\, \mathrm{SL}_3\,\mathbb{C} \to U_c$ is the map with $O_x(g)=gx$, then we get a map $d(O_x)_e \, : \,
\mathfrak{sl}_3 \to T_x U_c$, and $\mathfrak{sl}_3 (x):= \{ d(O_x)_e(\xi )\, | \, \xi \in \mathfrak{sl}_3\}$.
Hence $X_c$ is smooth. \\Assume $X_c$ were reducible, let $X_1$ and $X_2$ be two irreducible components. By prop.
2.1.2 (2) and the irreducibility of the group $\mathrm{SO}_3\,\mathbb{C}$, $X_1$ and $X_2$ are $(\mathrm{SL}_3\,
\mathbb{C} ,\mathrm{SO}_3\,\mathbb{C})$-sections of $U_c$, so we can find $g\in\mathrm{SL}_3\,\mathbb{C}$, $x_1\in
X_1$, $x_2\in X_2$, such that $g x_1 =x_2$. But then, by the $\mathrm{SL}_3\,\mathbb{C}$-equivariance of $f_{c,0}$,
$g$ stabilizes
$\mathbb{C}h_0$ and is thus in $\mathrm{SO}_3\,\mathbb{C}$. But, again by the irreducibility of
$\mathrm{SO}_3\,\mathbb{C}$, $x_2$ is also a point of $X_1$, i.e. $X_1$ and $X_2$ meet. This contradicts the
smoothness of $X_c$.
\end{proof}
The trouble is that, if $\overline{X}_c$ is the closure of $X_c$ in $\mathbb{P}(V(0,4))$, then $\overline{X}_c$ is
an irreducible component of the intersection of $5$ quintics. To eventually prove rationality, however, we would
like to have some equations of lower degree. This can be done for special $c$. 

\subsection{From quintic to cubic equations}
If $\Gamma_{f_c} \subset \mathbb{P} V(0,4)\times \mathbb{P} V(0,2)$ is the graph of $f_c$, it is natural to look
for $\mathrm{SL}_3\,\mathbb{C}$-equivariant maps
\begin{gather*}
\vartheta : V(0,4)\times V(0,2) \to V'
\end{gather*}
where $V'$ is another $\mathrm{SL}_3\,\mathbb{C}$-representation, $\vartheta$ is a homogeneous polynomial map in
both factors $V(0,4)$, $V(0,2)$, \emph{of low degree}, say $d$, in the first factor, \emph{linear} in the second,
and such that $\Gamma_{f_c}$ is an irreducible component of $\{ (x,y)\in \mathbb{P}V(0,4) \times \mathbb{P} V(0,2)
\, | \,
\vartheta (x , y)=0 \}$. If $V'$ is irreducible, there is an easy way to tell if $\vartheta$ vanishes on
$\Gamma_{f_c}$ \emph{for some} $c\in\mathbb{P}^1$: This will be the case if $V'$ occurs with multiplicity
one in $\mathrm{Sym}^{d+5}\, V(0,4)$. Here is the result.
\begin{definition}
Let $\Psi : V(0,4) \to V(2,2)$ be the up to factor unique $\mathrm{SL}_3\,\mathbb{C}$-equivariant, homogeneous of
degree $3$ polynomial map with the indicated source and target spaces, and let $\Phi : V(2,2)\times V(0,2) \to
V(2,1)$ be the up to factor unique bilinear $\mathrm{SL}_3\,\mathbb{C}$-equivariant map. Define $\Theta :
V(0,4)\times V(0,2)\to V(2,1)$ by $\Theta (x,y):=\Phi (\Psi (x),y)$.
\end{definition}

\begin{remark}
The existence and essential uniqueness of the maps of definition 2.3.1 can be easily deduced from known (and
implemented in
\ttfamily Magma\rmfamily ) decomposition laws for $\mathrm{SL}_3\,\mathbb{C}$-representations. That they are only
determined up to a nonzero constant factor will never bother us, and we admit this ambiguity in notation. The
explicit form of $\Psi$, $\Phi$, $\Theta$ will be needed later for checking certain non-degeneracy  conditions
through explicit computation. They can be found in Appendix A, formulas (64), (65).
\end{remark}
\begin{theorem}
\begin{itemize}
\item[(1)]
The linear map $\Theta (f, \cdot ) : V(0,2) \to V(2,1)$ has one-dimensional kernel for $f$ in an open dense subset
$V_0$ of $V(0,4)$, and, in particular, $\mathrm{ker}\,\Theta (h_0^2,\cdot )=\mathbb{C} h_0$.
\item[(2)]
For some $c_0\in\mathbb{P}^1$, $\Gamma_{f_{c_0}}$ is an irreducible component of $\{\Theta (x, y)=0\}\subset
\mathbb{P} V(0,4)\times \mathbb{P} V(0,2)$. 
\item[(3)]
$\overline{X_{c_0}}\subset \mathbb{P} V(0,4)$ coincides with the closure $\overline{X}$ in $\mathbb{P} V(0,4)$ of
the preimage
$X$ of
$h_0$ under the morphism from $\mathbb{P} V_0 \to \mathbb{P} V(0,2)$ given by $f\mapsto \mathrm{ker}\,\Theta
(f,\cdot )$, and is thus an irreducible component of the algebraic set $\{ \mathbb{C} f \, | \, \Phi ( \Psi (f) ,
h_0)=0
\}\subset \mathbb{P} V(0,4)$ defined by $15$ cubic equations. 
\item[(4)]
The rational map $\Psi : \mathbb{P} V(0,4)\dasharrow \overline{\Psi \mathbb{P} V(0,4)}\subset \mathbb{P}
V(2,2)$ as well as its restriction to $X$ are birational isomorphisms unto their images.
\end{itemize}
\end{theorem}
\begin{proof}
(1): One checks that $V(2,1)$ occurs with multiplicity one in the decomposition of $\mathrm{Sym}^8\, V(0,4)$. Thus
for some $c_0\in\mathbb{P}^1$, we have $\Theta (f, (c_{0,1}\alpha_1+c_{0,2}\alpha_2)(f))=0$ for all $f\in V(0,4)$. The
fact that $\mathrm{ker}\,\Theta (h_0^2,\cdot )=\mathbb{C} h_0$ follows from a direct computation using the
explicit form of $\Theta$. Thus, by
upper-semicontinuity, (1) follows.\\
(2): We have seen in (1) that $\Gamma_{f_{c_0}}$ is contained in $\{\Theta (x, y)=0\}$. Again by (1), 
\begin{gather*}
\Gamma_{f_{c_0}}\cap \left( (U_{c_0}\cap \mathbb{P} V_0 )\times \mathbb{P} V(0,2) \right) =\\ \{\Theta (x, y)=0\}
\cap \left( (U_{c_0}\cap \mathbb{P} V_0 )\times \mathbb{P} V(0,2) \right) \, ,
\end{gather*}
and (2) follows.\\
(3) follows from to (2) and the definition of $X_{c_0}$.\\
(4): Since $X$ is an $(\mathrm{SL}_3\,\mathbb{C} , \mathrm{SO}_3\,\mathbb{C})$-section of $\mathbb{P} V_0$, it
suffices to prove that the $\mathrm{SL}_3\,\mathbb{C}$-equivariant rational map $\Psi : \mathbb{P}
V(0,4)\dasharrow \overline{\Psi \mathbb{P} V(0,4)}$ (defined e.g. in the point $\mathbb{C}h_0^2$) is birational.
We will do this by writing down an explicit rational inverse. To do this, remark that $V(a,b)$ sits as an
$\mathrm{SL}_3\,\mathbb{C}$-invariant linear subspace inside $\mathrm{Sym}^a\mathbb{C}^3\otimes
\mathrm{Sym}^b(\mathbb{C}^3)^{\vee }$ (it has multiplicity one in the decomposition into irreducibles), thus
elements of
$V(a,b)$ may be viewed as tensors
$x=(x_{j_1 ,\dots , j_a}^{i_1 , \dots , i_b})\in T^b_a\, \mathbb{C}^3$, covariant of order $b$ and contravariant
of order $a$, or of type $b\choose a$. The inverse of the determinant tensor $\det^{-1}$ is thus in
$T^0_3\mathbb{C}^3$. For $f\in V(0,4)$ and $g\in V(2,2)$ one defines a bilinear
$\mathrm{SL}_3\,\mathbb{C}$-equivariant map $\alpha : V(0,4)\times V(2,2)\to \mathrm{Sym}^2\mathbb{C}^3\otimes
\mathrm{Sym}^3 (\mathbb{C}^3)^{\vee }$, $(f, g)\mapsto \alpha (f, g)$, as the contraction
\begin{gather*}
s^{i_1\, i_2\, i_3}_{j_1\, j_2}:= f^{i_1\, i_2\, i_4\, i_5}g_{i_5\, j_1}^{i_6\, i_3}\mathrm{det}^{-1}_{j_2\, i_4\,
i_6}\, ,
\end{gather*}
followed by the symmetrization map. One checks that $\mathrm{Sym}^2\mathbb{C}^3\otimes
\mathrm{Sym}^3 (\mathbb{C}^3)^{\vee }$ decomposes as $V(2,3)\oplus V(1,2) \oplus V(0,1)$, but $\mathrm{Sym}^4
V(0,4)$ does not contain these as subrepresentations (use \ttfamily Magma\rmfamily ), so $\alpha (f , \Psi (f))=0$
for all $f\in V(0, 4)$. But the explicit forms of $\Psi$ and $\alpha$ show that $\mathrm{ker} \,\alpha (\cdot , \Psi
(h_0^2))= \mathbb{C}h_0^2$, whence, by upper-semicontinuity, the dimension of the kernel of $\alpha (\cdot , \Psi
(f))$ is one for all $f$ in a dense open subset of $V(0,4)$, and the rational map $\Psi : \mathbb{P} V(0, 4)
\dasharrow \overline{\Psi \mathbb{P} V(0,4)}\subset \mathbb{P} V(2 ,2)$ has the rational inverse $\Psi (f) \mapsto
\mathrm{ker} \, \alpha (\cdot , \Psi (f))$.
\end{proof}
\begin{remark}
It would probably be illuminating to have a geometric interpretation of the covariant $\Psi : V(0,4)\to
V(2,2)$ given above similar to the one for $\alpha_1$, $\alpha_2$ in subsection 2.2. Though there is a
huge amount of classical projective geometry attached to plane quartics, I have been unable to find such a
geometric description.\\
Clearly, $\Psi$ vanishes on the cone of dominant vectors in $V(0,4)$, and one may check, using the explicit
formula for $\Psi$ in Appendix A (64), that $\Psi$ also vanishes on the $\mathrm{SL}_3\,\mathbb{C}$-orbit of
the degree $4$ forms in two variables, $x$ and $y$, say. However, this, together with the fact that $\Psi$ is of degree $3$, is not enough to characterize $\Psi$ since
the same holds also for e.g. the Hessian covariant.
\end{remark}

\subsection{From cubic to quadratic equations}
We have to fix some further notation.
\begin{definition}
\begin{itemize}
\item[(1)]
$Z$ is the affine cone in $V(2,2)$ over $\overline{\Psi (X)}\subset \mathbb{P} V(2,2)$. 
\item[(2)]
$L$ is the linear subspace $L:=\{ g\in V(2,2)\, |\, \Phi (g , h_0) =0 \} \subset V(2,2)$.
\item[(3)]
$\epsilon : V(0,4) \times V(0,2)\to V(2,2)$ is the unique (up to a nonzero factor) nontrivial
$\mathrm{SL}_3\,\mathbb{C}$-equivariant bilinear map with the indicated source and target spaces (the explicit
form is in Appendix A (66)).
\item[(4)]
$\zeta : V(0,4)\times V(0,2) \to V(1,1)$ is the unique (up to factor) nontrivial
$\mathrm{SL}_3\,\mathbb{C}$-equivariant map with the property that it is homogeneous of degree $2$ in both factors
of its domain (cf. Appendix A (67) for the explicit description). We put $\Gamma : =\zeta (\cdot , h_0) : V(0,4)
\to V(1,1)$. 

\end{itemize}
\end{definition}
Let us state explicitly what we are heading towards:
\begin{quote}
The affine cone $Z$ over the birational modification $\overline{\Psi (X)}$ of our $(\mathrm{SL}_3\, \mathbb{C} ,
\mathrm{SO}_3\,\mathbb{C})$-section $X\subset \mathbb{P} V_0\subset \mathbb{P} V(0,4)$ (whose closure in
$\mathbb{P}V(0,4)$ was seen to be an irreducible component of an algebraic set defined by $15$ \emph{cubic}
equations) has the following wonderful properties: $Z$ lies in $L$, the linear map $\epsilon (\cdot , h_0) : V(0,4)
\to V(2,2)$ restricts to an $\mathrm{SO}_3\mathbb{C}$-equivariant isomorphism between $V(0,4)$ and $L$, and if,
via this isomorphism, we transport $Z$ into $V(0,4)$ and call this $Y$, then the equations for $Y$ are given by
$\Gamma$! More precisely, $Y$ is the unique irreducible component of $\Gamma^{-1}(0)$ passing through the point
$h_0^2$, and $\Gamma$ maps $V(0,4)$ into a five-dimensional $\mathrm{SO}_3\,\mathbb{C}$-invariant subspace of
$V(1,1)$! 
\end{quote}
Thus, if we have carried out this program, $Y$ (or $Z$) will be proven to be an irreducible component of an
algebraic set defined by $5$ \emph{quadratic equations}! This seems quite miraculous, but a satisfactory
explanation why this happens probably requires an answer to the problem raised in remark 2.3.4.\\
We start with some preliminary observations: It is clear that $Z\subset L$ and $\mathbb{C} (\mathbb{P}
V(0,4))^{\mathrm{SL}_3\,\mathbb{C}} \simeq \mathbb{C} (Z)^{\mathrm{SO}_3\,\mathbb{C} \times \mathbb{C}^{\ast}}$,
$\mathbb{C}^{\ast }$ acting by homotheties. In the following, we need the decomposition into irreducibles of
$\mathrm{SL}_3\,
\mathbb{C}$-modules such as $V(2,2)$, $V(2,1)$, $V(1,1)$ and $V(0,4)$ as $\mathrm{SO}_3\,\mathbb{C}$-modules. The
patterns according to which irreducible representations of a complex semi-simple algebraic group decompose when
restricted to a smaller semi-simple subgroup are generally known as \emph{branching rules}. In our case the answer
is
\begin{gather}
V(2,2)= V(2,2)_8 \oplus V(2,2)_6 \oplus V(2,2)_4 \oplus V(2,2)'_4 \oplus V(2,2)_0\, ,\\
V(2,1)=V(2,1)_6 \oplus V(2,1)_4 \oplus V(2,1)_2 \, ,\\
V(1,1)=V(1,1)_4\oplus V(1,1)_2 \, , \\
V(0,4)=V(0,4)_8\oplus V(0,4)_4 \oplus V(0,4)_0\, .
\end{gather}
Here the subscripts indicate the numerical label of the highest weight of the respective
$\mathrm{SO}_3\,\mathbb{C}$-submodule of the ambient $\mathrm{SL}_3\,\mathbb{C}$-module under consideration.
Note also that $\mathrm{SO}_3\,\mathbb{C}\simeq \mathrm{PSL}_2\,\mathbb{C}$, so we are really back in the much
classically studied theory of \emph{binary forms}. It is not difficult (and fun) to check (12), (13), (14), (15) by
hand; let us briefly digress on how this can be done (cf. \cite{Fu-Ha}):\\
We fix the following notation. Let first $n=2l+1$ be an odd integer, $\mathfrak{g}=\mathfrak{sl}_n \,\mathbb{C}$
the Lie algebra of $\mathrm{SL}_n \,\mathbb{C}$, and let $\mathfrak{t}_{\mathfrak{g}}$ its standard torus of
diagonal matrices of trace $0$, and define the standard weights $\epsilon_i \in
\mathfrak{t}_{\mathfrak{g}}^{\vee}$, $i=1, \dots , n$, by $\epsilon_i (\mathrm{diag}(x_1, \dots , x_n)):= x_i$.
Inside $\mathfrak{g}$ we find $\mathfrak{h}:=\mathfrak{so}_n\,\mathbb{C}$ defined by
\begin{gather*}
\mathfrak{h}:=\left\{ \left( \begin{array}{ccc} X & Y & U \\ Z & -X^t & V \\ -V^t & -U^t & 0 \end{array}\right) \,
 |
\, X,\: Y,\: Z \in\mathfrak{gl}_l\,\mathbb{C} , \: Y^t=-Y^t ,\right. \\ \left. \: Z=-Z^t ,\: U, V\in\mathbb{C}^l
\right\}
\, .
\end{gather*}
Then $\mathfrak{t}_{\mathfrak{h}}:=\{ \mathrm{diag}(x_1, \dots , x_l , -x_1, \dots , -x_l, 0)\, | \, x_i\in\mathbb{C}
\}$; by abuse of notation we denote the restrictions of the functions $\epsilon_i$ to
$\mathfrak{t}_{\mathfrak{h}}$ by the same letters. The fundamental weights of $\mathfrak{g}$ are $\pi_i:=
\epsilon_1+\dots +\epsilon_i$, $i=1, \dots , n-1$, the fundamental weights of $\mathfrak{h}$ are
$\omega_i:=\epsilon_1+\dots +\epsilon_i$, ($1\le i\le l-1$) and $\omega_l:=(\epsilon_1+\dots +\epsilon_l)/2$. Let
$\Lambda_{\mathfrak{g}}$ and $\Lambda_{\mathfrak{h}}$ be the corresponding weight lattices.
$\Lambda_{\mathfrak{g}}^{+}$ and
$\Lambda_{\mathfrak{h}}^{+}$ are the dominant weights. For $\mathfrak{g}$ (and similarly for $\mathfrak{h}$) an
irreducible representation $V(\lambda )$ for $\lambda \in\Lambda_{\mathfrak{g}}^{+}$ comes with its
\emph{formal character}
\begin{gather*}
\mathrm{ch}_{\lambda }:=\sum_{\mu\in\Pi (\lambda )} m_{\lambda }(\mu ) e(\mu) \in\mathbb{Z} [\Lambda_{\mathfrak{g}}
]
\, ,
\end{gather*} an element of the group algebra $\mathbb{Z}[\Lambda_{\mathfrak{g}} ]$ generated by the symbols $e
(\lambda)$ for
$\lambda\in\Lambda_{\mathfrak{g}}$, where $\Pi (\lambda )$ means the weights of $V(\lambda )$, and $m_{\lambda
}(\mu )$ is the dimension of the weight space corresponding to $\mu$ in $V(\lambda )$. We have a formal character
$\mathrm{ch}_V$ for any finite-dimensional $\mathfrak{g}$-module $V=V(\lambda_1)\oplus \dots \oplus V(\lambda_t)$,
$\lambda_1 , \dots , \lambda_t\in \Lambda^{+}_{\mathfrak{g}}$ defined by
\begin{gather*}
\mathrm{ch}_{V}:=\sum_{i=1}^{t} \mathrm{ch}_{\lambda_i} \, .
\end{gather*}
The important point is that $V$ (i.e. its irreducible constituents) can be recovered from the formal character
$\mathrm{ch}_V$, meaning that in $\mathbb{Z}[\Lambda_{\mathfrak{g}}]$ we can write $\mathrm{ch}_V$ uniquely as a
$\mathbb{Z}$-linear combination of characters corresponding to dominant weights
$\lambda\in\Lambda^{+}_{\mathfrak{g}}$.\\
We go back to the case $l=1$, $n=3$. We have
$\mathfrak{h}=\mathfrak{so}_3\,\mathbb{C}=\mathfrak{sl}_2\,\mathbb{C}$. The character $\mathrm{ch}_{V(a)}$ of the
irreducible $\mathfrak{so}_3\,\mathbb{C}$-module
$V(a):=V(a\omega_1)$ is not hard: The weights of $V(a)$ are
\begin{gather*}
-a\omega_1 , \: (-a+2)\omega_1 , \: \dots , \: (a-2)\omega_1 , \: a\omega_1 \, 
\end{gather*}
(all multiplicities are $1$). It remains to understand the weights and their multiplicities in the irreducible
$\mathfrak{g}=\mathfrak{sl}_3\,\mathbb{C}$-module $V(a,b):=V(a\pi_1+ b\pi_2)$. In fact noting that $\pi_1$
restricted to the diagonal torus of $\mathfrak{so}_3\,\mathbb{C}$ above is $2\omega_1$, and the restriction of
$\pi_2$ is $0$, we see that, once we know the formal character of $V(a,b)$ as
$\mathfrak{sl}_3\,\mathbb{C}$-module, we simply substitute $2\omega_1$ for $\pi_1$ and $0$ for $\pi_2$ in the
result and obtain in this way the formal character of the $\mathfrak{so}_3\,\mathbb{C}$\emph{-module} $V(a,b)$,
and hence its decomposition into irreducible constituents as $\mathfrak{so}_3\,\mathbb{C}$-module.\\
Let us assume $a\ge b$ (otherwise pass to the dual representation); we describe the weights and their
multiplicities of the $\mathfrak{sl}_3\,\mathbb{C}$-module $V(a,b)$ following \cite{Fu-Ha}, p. 175ff.: Imagine a
plane with a chosen origin from which we draw two vectors of unit length, representing $\pi_1$ and $\pi_2$, such
that the angle measured counterclockwise from $\pi_1$ to $\pi_2$ is $60^{\circ}$. Thus the points of the lattice
spanned by $\pi_1$, $\pi_2$ are the vertices of a set of equilateral congruent triangles which gives a tiling of
the plane.\\
The weights of $V(a,b)$ are the lattice points which lie on the edges of a sequence of $b$ (not necessarily
regular) hexagons $H_i$ with vertices at lattice points, $i=0,\dots , b-1$, and a sequence of
$[(a-b)/3]+1$ triangles $T_j$, $j=0,\dots ,[(a-b)/3]$. The $H_i$ and $T_j$ are concentric around the origin, and
$H_i$ has one vertex at $(a-i)\pi_1 + (b-i)\pi_2$, $T_j$ has one vertex at the point $(a-b-3j)\pi_1$, and $H_i$
and $T_j$ are otherwise determined by the condition that the lines through $\pi_1$, $\pi_2$, $\pi_2-\pi_1$ are
axes of symmetry for them, i.e. they are preserved by the reflections in these lines (one should make a picture
now).\\ The multiplicities of the weights obtained in this way are as follows: Weights lying on $H_i$ have
multiplicity
$i+1$, and weights lying on one of the $T_j$ have multiplicity $b$. This completely determines the formal
character of $V(a,b)$.\\
Let us look at $V(2,2)$ for example. Here we get three concentric regular hexagons (one of them is degenerate and
consists of the origin alone). The weights are thus:
\begin{gather*}
2\pi_1+2\pi_2, \: 3\pi_2 , \: -2\pi_1+4\pi_2 , \: -3\pi_1+3\pi_2 , \: -4\pi_1+2\pi_2 ,\: -3\pi_1 ,\\ \:
-2\pi_1-2\pi_2 , \: -3\pi_2 , \: 2\pi_1-4\pi_2 , \: 3\pi_1-3\pi_2 , \: 4\pi_1-2\pi_2 , \: 3\pi_1
\end{gather*}
(these are the ones on the outer hexagon, read counterclockwise, and have multiplicity one),
\begin{gather*}
\pi_1+\pi_2 , \: -\pi_1+2\pi_2 , \: -2\pi_1+\pi_2 , \: -\pi_1-\pi_2 , \: \pi_1-2\pi_2 , \: 2\pi_1-\pi_2
\end{gather*}
(these lie on the middle hexagon and have multiplicity two), and finally there is $0$ with multiplicity $3$
corresponding to the origin. Consequently, the formal character of $V(2,2)$ as a representation of
$\mathfrak{so}_3\,\mathbb{C}$ is
\begin{gather*}
e(-8\omega_1 )+ 2 e(-6\omega_1)+4e(-4\omega_1) + 4e(-2\omega_1) + 5 e(0\omega_1) \, ,\\
+4e(2\omega_1 )+ 4e(4\omega_1)+2e(6\omega_1) + e(8\omega_1)
\end{gather*}
which is equal to $\mathrm{ch}_{V(8)}+\mathrm{ch}_{V(6)}+2\mathrm{ch}_{V(4)}+\mathrm{ch}_{V(0)}$. This proves
(12), and (13), (14) and (15) are similar.\\
We resume the discussion of the main content of subsection 2.4. Before stating the main theorem, we collect some
preliminary facts in the following lemma.
\begin{lemma}
\begin{itemize}
\item[(1)]
The following deccomposition of $L\subset V(2,2)$ as $\mathrm{SO}_3\,\mathbb{C}$-subspace of $V(2,2)$ holds
(possibly after interchanging the roles of $V(2,2)_4$ and $V(2,2)_4'$):
\begin{gather*}
L = V(2,2)_8\oplus V(2,2)_4 \oplus V(2,2)_0\: .
\end{gather*}
\item[(2)]
The map $\epsilon (\cdot , h_0) : V(0,4) \to V(2,2)$ is an $\mathrm{SO}_3$-equivariant isomorphism onto $L$.
\item[(3)]
Putting $Y:=\epsilon (\cdot , h_0)^{-1}(Z)\subset V(0,4)$, we have $h_0^2\in Y$.
\item[(4)]
One has $\Gamma (V(0,4))\subset V(1,1)_4\subset V(1,1)$, and the inclusion $Y \subset \Gamma^{-1}(0)$ holds.
\end{itemize}
\end{lemma}
\begin{proof}
(1): Using the explicit form of $\Phi$ one calculates that the dimension of the image of the 
$\mathrm{SO}_3\,\mathbb{C}$-equivariant map $\Phi (\cdot , h_0) : V(2,2) \to V(2,1)$ is $12$. Thus, in view of the
decomposition (13) of $V(2,1)$ as $\mathrm{SO}_3\,\mathbb{C}$-representation, we must have $\Phi (V(2,2) , h_0)
=V(2,1)_6\oplus V(2,1)_4$. Since
\begin{gather}
\dim V(a,b)=\frac{1}{2}(a+1)(b+1)(a+b+2)\, ,
\end{gather}
the dimension of $V(2,2)$ is $27$ and the kernel $L$ of $\Phi (\cdot , h_0)$ has dimension $15$; in fact,
$V(2,2)_8$, $V(2,2)_0$ and (after possibly exchanging $V(2,2)_4$ and $V(2,2)_4'$) $V(2,2)_4$ must all be in the
kernel, since these representations do not appear in the decomposition of the image.\\
(2): Using the explicit form of $\epsilon$ given in Appendix A (66), one calculates that the dimension of the
image of $\epsilon (\cdot , h_0)$ is $15$ whence this linear map is injective. Moreover, its image is contained in
$L$, hence equals $L$, because the map $V(0,4)\times V(0,2) \to V(2,1)$ given by $(f,g)\mapsto \Phi (\epsilon
(f,g), g)$ is identically zero since there is no $V(2,1)$ in the decomposition of $V(0,4)\otimes\mathrm{Sym}^2\,
V(0,2)$.\\
(3): As we saw in theorem 2.3.3 (1), $\mathbb{C}h_0^2 \in X$, and we have $0\neq \Psi (h_0^2)\in Z$. From the
decomposition (12), we get, $\Psi (h_0^2)$ being invariant, $\langle \Psi (h_0^2)
\rangle_{\mathbb{C}}=L^{\mathrm{SO}_3\,\mathbb{C}}$. By the decomposition (15), we get that the preimage under
$\epsilon (\cdot ,h_0)$ of $\Psi (h_0^2)$ spans the $\mathrm{SO}_3\,\mathbb{C}$-invariants $V(0,4)_0$ which are
thus in $Y$. So in particular, $h_0^2\in Y$.\\
(4): The first part is straightforward: Just decompose $\mathrm{Sym}^2\, V(0,4)$ as
$\mathrm{SO}_3\,\mathbb{C}$-module by the methods explained above, and check that it does not contain any
$\mathrm{SO}_3\,\mathbb{C}$-submodule the highest weight of which has numerical label $2$ (this suffices by (14)).
The second statement of (4) follows from the observation that the map $\zeta : V(0,4) \times V(0,2) \to V(1,1)$
(Def. 2.4.1 (4)) factors:
\begin{gather*}
c \cdot\zeta = \tilde{\gamma}\circ \epsilon \, , \: c\in\mathbb{C}^{\ast}\, ,
\end{gather*}
where $\tilde{\gamma} : V(2,2) \to V(1,1)$ is the unique (up to nonzero scalar) non-trivial
$\mathrm{SL}_3\,\mathbb{C}$-equivariant map which is homogeneous of degree $2$. This is because $V(1,1)$ occurs in
the decomposition of $\mathrm{Sym}^2\, V(0,4)\otimes \mathrm{Sym}^{2}\, V(0,2)$ with multiplicity one, and
$\tilde{\gamma}\circ \epsilon$ is not identically zero, as follows from the explicit form of these maps (cf.
Appendix A, (66), (68)). Thus, defining $\tilde{\Gamma} : V(0,4)
\to V(1,1)$ by $\tilde{\Gamma}(\cdot ):=(\tilde{\gamma}\circ \epsilon )(\cdot ,h_0)$ (which thus differs from
$\Gamma $ just by a nonzero scalar), we must show $\tilde{\Gamma }(Y)=0$. But recalling the definitions of $Y$,
$\tilde{\Gamma}$ and $Z$ (Def. 2.4.1 (1)), it suffices to show that $\tilde{\gamma}\circ \Psi $ is identically
zero; the latter is true since it is an $\mathrm{SL}_3\,\mathbb{C}$-equivariant map from $V(0,4)$ to $V(1,1)$,
homogeneous of degree $6$, but $\mathrm{Sym}^6\, V(0,4)$ does not contain $V(1,1)$. This proves (4).
\end{proof}
Let us now pass from $\mathrm{SO}_3\,\mathbb{C}$ to the $\mathrm{PSL}_2\,\mathbb{C}$-picture and denote by $V(d)$
the space of binary forms of degree $d$ in the variables $z_1$, $z_2$. This is of course consistent with our
previous notation since, under the isomorphism $\mathfrak{so}_3\,\mathbb{C}\simeq \mathfrak{sl}_2\,\mathbb{C}$,
$V(d)$ is just the irreducible $\mathfrak{so}_3\,\mathbb{C}$-module the highest weight of which has numerical
label $d$; since we consider $\mathrm{PSL}_2\,\mathbb{C}$-representations, $d$ is always even.\\
We will fix a covering $\mathrm{SL}_2\,\mathbb{C} \to \mathrm{SO}_3\,\mathbb{C}$ and thus an isomorphism
$\mathrm{PSL}_2\,\mathbb{C}$ $\simeq\mathrm{SO}_3\,\mathbb{C}$, and we will fix isomorphisms $\delta_1 : V(0)\oplus
V(4)
\oplus V(8) \to V(0,4)$ and $\delta_2 : V(4) \to V(1,1)_4$ such that $(1,0,0)$ maps to $h_0^2$ under $\delta_1$
and both $\delta_1$ and $\delta_2$ are equivariant with respect to the isomorphism
$\mathrm{PSL}_2\,\mathbb{C}\simeq\mathrm{SO}_3\,\mathbb{C}$; we will discuss in a moment how this is done, but
for now this is not important. Look at the diagram\\
\vspace{0.5cm}
\setlength{\unitlength}{1cm}
\begin{picture}(8,6)
\put(8,4){$V(0)\oplus V(4)\oplus V(8)$}
\put(0,4){$Y\subset \Gamma^{-1}(0)\subset V(0,4)$}
\put(7,4){\vector(-1,0){2}}
\put(8.7,1.5){$V(4)$}
\put(7,1.6){\vector(-1,0){2}}
\put(1.9,1.5){$0\in V(1,1)_4$}
\put(3,0.8){$\cap$}
\put(2.8,0.3){$V(1,1)\simeq V(1,1)_4\oplus V(1,1)_2$}
\put(1.9,3.5){\vector(0,-1){1.5}}
\put(3.2,3.5){\vector(0,-1){1.5}}
\put(8.9,3.5){\vector(0,-1){1.5}}

\put(6,3.5){$\delta_1$}
\put(6,4.1){$\simeq$}
\put(6,1.2){$\delta_2$}
\put(6,1.7){$\simeq$}
\put(8.7, 2.8){$\delta :=\delta_2^{-1}\circ\Gamma\circ\delta_1$}
\put(0.5,2.8){$\Gamma |_{\Gamma^{-1}(0)}$}
\put(3.5,2.8){$\Gamma$}

\put(9.8,4.5){$\cap$}
\put(9.8,5){$U:=\delta_1^{-1}(Y)$}

\put(7.6,4.7){$(1,0,0)$}
\put(0,4.7){$h_0^2$}
\put(7,4.8){\vector(-1,0){6}}
\put(7,4.7){\line(0,1){0.2}}
\put(4,5){$\delta_1$}
\end{picture}\\
By part (4) of lemma 2.4.2, we have $\delta^{-1}(0)\supset U$, and by part (3) of the same lemma, $(1,0,0)\in U$.
Moreover, recalling our construction of $X$ in theorem 2.3.3, we see that $\dim X=\dim \mathbb{P}\, V(0,4) -\dim
\mathbb{P}\, V(0,2)=14-5=9$, whence, chasing through the definitions of $Z$, $Y$, $U$, we get $\dim U=10$. But the
explicit form of $\delta$ (we will see this in a moment) allows us to conclude, by explicit calculation of the
rank of the differential of $\delta$ at the invariant point $(1,0,0)$, that $\dim\, T_{(1,0,0)}\, U =10$, whence
$T_{(1,0,0)}\, U= V(0)\oplus V(8)$. Therefore, as $U$ is irreducible, it is the unique component of the (possibly
reducible) variety $\delta^{-1}(0)$ passing through $(1,0,0)$. Moreover, it is clear the condition $\{\delta =0\}$
amounts to $5$ quadratic equations! We have proven
\begin{theorem}
There is an isomorphism
\begin{gather}
\mathbb{C}(\mathbb{P}\, V(0,4))^{\mathrm{SL}_3\,\mathbb{C}}\simeq \mathbb{C}
(U)^{\mathrm{PSL}_2\,\mathbb{C}\times\mathbb{C}^{\ast}}
\end{gather}
where
\begin{gather*}
\delta : V(0)\oplus V(4) \oplus V(8) \to V(4)
\end{gather*}
is $\mathrm{PSL}_2$-equivariant and homogeneous of degree $2$, and 
 $U$ is the unique irreducible component of $\delta^{-1}(0)$ passing through $(1,0,0)$. Moreover, $\dim U=10$
and $T_{(1,0,0)}\, U = V(0)\oplus V(8)$.
\end{theorem}
We close this section by describing the explicit form of the covering
$\mathrm{SL}_2\,\mathbb{C}\to\mathrm{SO}_3\,\mathbb{C}$ and the maps $\delta_1$, $\delta_2$, and by making some
remarks on transvectants and the final formula for the map $\delta$.\\
Let $e_1,\: e_2 , \:  e_3$ be the standard basis in $\mathbb{C}^3$, and denote by $x_1,\: x_2, \: x_3$ the dual
basis in $(\mathbb{C}^3)^{\vee }$. In this notation, $h_0^2=x_1x_3-x_2^2$. We may view the $x's$ as coordinates on
$\mathbb{C}^3$ and identify $\mathbb{C}^3$ with the Lie algebra $\mathfrak{sl}_2\,\mathbb{C}$ by assigning to
$(x_1, x_2, x_3)$ the matrix
\begin{gather*}
X=\left( \begin{array}{cc} x_2 & -x_1 \\ x_3 & -x_2 \end{array}\right) \in \mathfrak{sl}_2\,\mathbb{C} \, .
\end{gather*}
Consider the adjoint representation $\mathrm{Ad}$ of $\mathrm{SL}_2\,\mathbb{C}$ on $\mathfrak{sl}_2\,\mathbb{C}$.
Clearly, for $X\in \mathfrak{sl}_2\,\mathbb{C}$, $A\in \mathrm{SL}_2\,\mathbb{C}$, the map $\mathrm{Ad}(A) \, :\,
X\mapsto AXA^{-1}$ preserves the determinant of $X$, which is just our $h_0$; the kernel of $\mathrm{Ad}$ is the
center $\{ \pm 1 \}$ of $\mathrm{SL}_2\,\mathbb{C}$, and since $\mathrm{SL}_2\,\mathbb{C}$ is connected, the image
of 
$\mathrm{Ad}$ is $\mathrm{SO}_3\,\mathbb{C}$. This is how we fix the isomorphism $\mathrm{PSL}_2\,\mathbb{C}\simeq
\mathrm{SO}_3\,\mathbb{C}$ explicitly, and how we view $\mathrm{SO}_3\,\mathbb{C}$ as a subgroup of
$\mathrm{SL}_3\,\mathbb{C}$. Note that the induced isomorphism
$\mathfrak{sl}_2\,\mathbb{C}\to\mathfrak{so}_3\,\mathbb{C}$ on the Lie algebra level can be described as follows:
\begin{gather}
e:=\left( \begin{array}{cc} 0 & 1 \\ 0 & 0 \end{array}\right) \mapsto \left( \begin{array}{ccc} 0 & 2 & 0 \\ 0 & 0
& 1\\ 0 & 0 & 0 \end{array}\right)\, ,\\
f:=\left( \begin{array}{cc} 0 & 0 \\ 1 & 0 \end{array}\right) \mapsto \left( \begin{array}{ccc} 0 & 0 & 0 \\ 1 & 0
& 0\\ 0 & 2 & 0 \end{array}\right)\, , \nonumber \\
h:=\left( \begin{array}{cc} 1 & 0 \\ 0 & -1 \end{array}\right) \mapsto \left( \begin{array}{ccc} 2 & 0 & 0 \\ 0 & 0
& 0\\ 0 & 0 & -2 \end{array}\right) \nonumber
\end{gather}
(where we view $\mathfrak{so}_3\,\mathbb{C}$ as a subalgebra of $\mathfrak{sl}_3\,\mathbb{C}$ in a way consistent
with the inclusion on the group level described above). For example,
\begin{gather*}
\mathrm{ad}\left( \left( \begin{array}{cc} 0 & 1 \\ 0 & 0 \end{array}\right) \right) (X)= \left(
\begin{array}{cc} 0 & 1 \\ 0 & 0 \end{array}\right)\left( \begin{array}{cc} x_2 & -x_1 \\ x_3 & -x_2
\end{array}\right) \\ - \left(
\begin{array}{cc} x_2 & -x_1 \\ x_3 & -x_2 \end{array}\right)\left(
\begin{array}{cc} 0 & 1 \\ 0 & 0 \end{array}\right) = \left( \begin{array}{cc} x_3 & -2x_1 \\ 0 & -x_3
\end{array}\right) \, ,
\end{gather*}
so
\begin{gather*}
\left( \begin{array}{c} x_1\\ x_2\\ x_3 \end{array}\right) \mapsto \left( \begin{array}{ccc} 0 & 2 & 0 \\ 0 & 0 &
1\\ 0 & 0 & 0 \end{array}\right) \left( \begin{array}{c} x_1 \\ x_2 \\ x_3 \end{array}\right)\, .
\end{gather*}
To give the isomorphism $\delta_1 : V(0) \oplus V(4) \oplus V(8) \to V(0,4)$ explicitly, we just have to find 
highest weight vectors inside $V(0)$, $V(4)$, $V(8)$ and corresponding highest weight vectors inside $V(0,4)$. For
example, $h$ acts on $z_2^4 \in V(4)$ by multiplication by $4$, and $z_2^4$ is killed by $e$, so this is a highest
weight vector inside $V(4)$. But if we compute
\begin{gather*}
h\cdot (x_1x_3^3-x_2^2x_3^2) = (h\cdot x_1)x_3^3 + 3 x_1 (h\cdot x_3) x_3^2 - 2 (h\cdot x_2)x_2 x_3^2\\ - 2x_2^2
(h\cdot x_3) x_3 = (-2 x_1)x_3^3 + 3 x_1 (2 x_3) x_3^2 - 2\cdot 0\cdot x_2 x_3^2\\ - 2x_2^2
(2 x_3) x_3=  4(x_1x_3^3-x_2^2x_3^2 ) \, \quad \mathrm{and}\\ 
e\cdot (x_1x_3^3-x_2^2x_3^2) = (e\cdot x_1)x_3^3 + 3 x_1 (e\cdot x_3) x_3^2 - 2 (e\cdot x_2)x_2 x_3^2\\ - 2x_2^2
(e\cdot x_3) x_3 = (-2x_2)\cdot x_3^3 + 3 x_1 \cdot 0\cdot  x_3^2 - 2(-x_3) x_2 x_3^2\\ - 2x_2^2
\cdot 0\cdot  x_3 =0
\end{gather*}
(use (18) and remark that the $x$'s are dual variables, so we have to use the dual action), then we find that a
corresponding highest weight vector for the submodule of $V(0,4)$ isomorphic to $V(4)$ is $x_1x_3^3-x_2^2x_3^2$.
Proceeding in this way, we see that we can define $\delta_1$ uniquely by the requirements:
\begin{gather}
\delta_1 \, : \, 1\mapsto h_0^2 \, , \: z_2^4 \mapsto x_1x_3^3-x_2^2x_3^2 \, , \: z_2^8 \mapsto x_3^4 \, ,
\end{gather}
and using the Lie algebra action and linearity, we can compute the values of $\delta_1$ on a set of basis vectors
in $V(0)\oplus V(4) \oplus V(8)$.\\
To write down $\delta_2$ explicitly, remark that $V(1,1)$ may be viewed as the
$\mathrm{SL}_3\,\mathbb{C}$-submodule of $\mathbb{C}^3\otimes (\mathbb{C}^3)^{\vee}$ consisting of those tensors
that are annihilated by
\begin{gather*}
\Delta := \frac{\partial }{\partial e_1}\otimes \frac{\partial}{\partial x_1} +
\frac{\partial }{\partial e_2}\otimes \frac{\partial}{\partial x_2} +
\frac{\partial }{\partial e_3}\otimes \frac{\partial}{\partial x_3} \, .
\end{gather*}
We take again our highest weight vector $z_2^4 \in V(4)$, and all we have to do is to find a vector in
$\mathbb{C}^3\otimes (\mathbb{C}^3)^{\vee}$ on which $h$ acts by multiplication by $4$ and which is annihilated by
$e$ and $\Delta$. Indeed, $e_1x_3$ is one such. Thus we define $\delta_2$ by 
\begin{gather*}
\delta_2 \, : \, z_2^4 \mapsto e_1x_3 \, .
\end{gather*}
Then it is easy to compute the values of $\delta_2$ on basis elements of $V(4)$ in the same way as for
$\delta_1$.\\
Let us recall the classical notion of \emph{transvectants} ("\"{U}berschiebung " in German). Let $d_1$, $d_2$, $n$
be nonnegative integers such that $0\le n \le \mathrm{min} (d_1 , d_2)$. For $f\in V(d_1)$ and $g\in V(d_2)$ one
puts
\begin{gather}
\psi_n( f,g) := \frac{(d_1-n) !}{d_1 !} \frac{(d_2-n) !}{d_2 !}\sum_{i=0}^n (-1)^i {n \choose i} \frac{\partial^n
f}{ \partial z_1^{n-i}\partial z_2^i} \frac{\partial^n g}{\partial z_1^i \partial z_2^{n-i} }
\end{gather}
(cf. \cite{B-S}, p. 122). The map $(f,g)\mapsto \psi_n(f,g)$ is a bilinear and
$\mathrm{SL}_2\,\mathbb{C}$-equivariant map from $V(d_1)\times V(d_2)$ onto $V(d_1 +d_2 -2n)$. The map
\begin{gather*}
V(d_1)\otimes V(d_2)\to \bigoplus_{n=0}^{\mathrm{min}(d_1 , d_2)} V(d_1+d_2-2n) \\
(f,g) \mapsto \sum_{n=0}^{\mathrm{min}(d_1 , d_2)} \psi_n(f,g) 
\end{gather*}
is an isomorphism of $\mathrm{SL}_2\, \mathbb{C}$-modules ("Clebsch-Gordan decomposition"). Thus transvectants
make the decomposition of $V(d_1)\otimes V(d_2)$ into irreducibles explicit; a similar result for
$\mathrm{SL}_3\,\mathbb{C}$-representations would be very important in several areas of computational invariant
theory and also for the rationality question for moduli spaces of plane curves, but is apparently unknown.\\
The explicit form of $\delta$ that results from the computations is then
\begin{gather}
\delta (f_0 , f_4 , f_8) = -\frac{6}{1225} \psi_6 (f_8 , f_8) + \frac{1}{840} \psi_4(f_8 , f_4)\\ +
\frac{11}{54}\psi_2(f_4 , f_4) - \frac{7}{36} f_4f_0 \, ,\nonumber 
\end{gather}
where $(f_0,f_4,f_8)\in V(0)\oplus V(4)\oplus V(8)$. Note that the fact that $\delta$ turns out to be such a linear
combination of transvectants is no surprise in view of the Clebsch-Gordan decomposition: In fact, $\delta$ may be
viewed as a map
\begin{gather*}
\delta ' : (V(0)\oplus V(4) \oplus V(8)) \otimes (V(0)\oplus V(4) \oplus V(8)) \to V(4)
\end{gather*}
and using the fact that $\delta$ is symmetric and collecting only those tensor products in the preceding formula
for which $V(4)$ is a subrepresentation, we see that $\delta$ comes from a map
\begin{gather*}
\delta '' : (V(0)\otimes V(4)) \oplus (V(4)\otimes V(4)) \\ \oplus (V(8)\otimes V(4)) \oplus (V(8)\otimes V(8)) \to
V(4) \, .
\end{gather*}
Thus it is clear from the beginning that $\delta$ will be a linear combination of $\psi_6$, $\psi_4$, $\psi_2$,
$\psi_0$ as in formula (21), and the actual coefficients are easily calculated once we know $\delta$ explicitly!\\
In fact, the next lemma shows that the actual coefficients of the transvectants $\psi_i$'s occurring in $\delta$
are not very important.
\begin{lemma}
For $\lambda :=(\lambda_0 , \lambda_2 , \lambda_4 , \lambda_6)\in\mathbb{C}^4$ consider the homogeneous of degree
$2$ 
$\mathrm{PSL}_2$-equivariant map
\begin{gather*}
\delta_{\lambda } : V(8)\oplus V(0) \oplus V(4) \to V(4) \\
f_8 + f_0 +f_4 \mapsto \lambda_6 \psi_6(f_8 , f_8) +2\lambda_4\psi_4 (f_8 , f_4) +\lambda_2 \psi_2(f_4, f_4)
+2\lambda_0 f_4f_0 \, .
\end{gather*}
Suppose that $\lambda_0\neq 0$. Then:
\begin{itemize}
\item[(1)]
One has $1\in \delta_{\lambda}^{-1}(0)$ and $T_{1}\, \delta_{\lambda}^{-1}(0)=V(8)\oplus V(0)$; thus there is a
unique irreducible component $U_{\lambda}$ of $\delta_{\lambda}^{-1}(0)$ passing through $1$ on which $1$ is a
smooth point.
\item[(2)]
If furthermore $\lambda\in (\mathbb{C}^{\ast})^4$, then $\mathbb{P} U_{\lambda}$ is
$\mathrm{PSL}_2\,\mathbb{C}$-equivariantly isomorphic to $\mathbb{P} U_{(1 , 6\epsilon , 1, 6)}$ for some
$\epsilon\neq 0$ (depending on $\lambda$).
\end{itemize}
\end{lemma}
\begin{proof}
Part (1) is a straightforward calculation, and for part (2) we choose complex numbers $\mu_0$, $\mu_4$, $\mu_8$
with the properties $6\mu_8^2=\lambda_6$, $\mu_4\mu_8 =\lambda_4$, $\mu_0\mu_4 =\lambda_0$, and compute $\epsilon$
from $6\epsilon\mu_4^2=\lambda_2$. Then the map from $\mathbb{P} U_{\lambda}$ to $\mathbb{P} U_{(1 , 6\epsilon , 1
,6)}$ given by sending $[ f_0+f_4+f_8]$ to $[ \mu_0 f_0 +\mu_4 f_4 + \mu_8 f_8]$ gives the desired isomorphism.
\end{proof}
In the next section we will see that for any $\epsilon\neq 0$, the $\mathrm{PSL}_2\,\mathbb{C}$-quotient of
$\mathbb{P} U_{(1 , 6\epsilon , 1 , 6)}$ is rational, and so the same holds for $\mathbb{P} U_{\lambda }$ for any
$\lambda\in (\mathbb{C}^{\ast })^4$; note however that the reduction step in lemma 2.4.4 (2) just simplifies the
subsequent calculations, but is otherwise not substantial.

\section{Further sections and inner projections}

\subsection{Binary quartics again and a $(\mathrm{PSL}_2\,\mathbb{C} , \mathfrak{S}_4)$-section}

All the subsequent constructions and calculations depend very much on the geometry of the
$\mathrm{PSL}_2\,\mathbb{C}$-action on the module $V(4)$. In fact, the first main point in the proof that
$\mathbb{P} U_{\lambda} / \mathrm{PSL}_2\,\mathbb{C}$ is rational will be the construction of a
$(\mathrm{PSL}_2\,\mathbb{C} , \mathfrak{S}_4)$-section of this variety ($\mathfrak{S}_4$ being the group of
permutations of $4$ elements); this is done by using proposition 2.1.2 (2) for the projection of
$V(8)\oplus V(0)\oplus V(4)$ to $V(4)$ and producing such a section for $V(4)$ via the concept of \emph{stabilizer
in general position} which we recall next.
\begin{definition}
Let $G$ be  a linear algebraic group $G$ acting on an irreducible variety $X$. A stabilizer in general position
(s.g.p.) for the action of $G$ on $X$ is a subgroup $H$ of $G$ such that the stabilizer of a general point in $X$
is conjugate to $H$ in $G$.
\end{definition}
An s.g.p. (if it exists) is well-defined to within conjugacy, but it need not exist in general; however, for the
action of a reductive group $G$ on an irreducible smooth affine variety, an s.g.p. always exists by results of
Richardson and Luna (cf. \cite{Po-Vi}, \S 7).
\begin{proposition}
For the action of $\mathrm{PSL}_2\,\mathbb{C}$ on $V(4)$, an s.g.p. is given by the subgroup $H$ generated by
\begin{gather*}
\omega := \left[ \left( \begin{array}{cc} 0 & 1 \\ -1 & 0 \end{array} \right) \right] \quad \mathrm{and} \quad 
\rho := \left[ \left( \begin{array}{cc} i & 0 \\ 0 & -i \end{array} \right) \right]\, .
\end{gather*}
$H$ is isomorphic to the Klein four-group $\mathfrak{V}_4\simeq \mathbb{Z}/2\mathbb{Z}
\oplus\mathbb{Z}/2\mathbb{Z}$ and its normalizer $N(H)$ in $\mathrm{PSL}_2\,\mathbb{C}$ is isomorphic to
$\mathfrak{S}_4$; one has $N(H)/H\simeq \mathfrak{S}_3$.\\
More explicitly, $N(H)=\langle \tau , \sigma\rangle $, where, putting $\theta := \mathrm{exp}(2\pi i /8)$, one has
\begin{gather*}
\tau := \left[ \left(\begin{array}{cc} \theta^{-1} & 0 \\ 0 & \theta \end{array}\right) \right]\, , \quad \sigma
:= \left[ \frac{1}{\sqrt{2}} \left( \begin{array}{cc} \theta^3 & \theta^7 \\ \theta^5 & \theta^5
\end{array}\right) \right] \, .
\end{gather*}
\end{proposition}

\begin{proof}
We will give a geometric proof due to Bogomolov (\cite{Bog1}, p.18). A general homogeneous degree $4$ binary form
$f\in V(4)$ determines a set of $4$ points $\Sigma \subset \mathbb{P}^1$; the double cover of $\mathbb{P}^1$ with
branch points $\Sigma$ is an elliptic curve; it is acted on by its subgroup of $2$-torsion points $H_f\simeq
\mathbb{Z}/2\mathbb{Z}\oplus \mathbb{Z}/2\mathbb{Z}$, and this action commutes with the sheet exchange map, hence
descends to an action of $H_f$ on $\mathbb{P}^1$ which preserves the point set $\Sigma$ and thus the polynomial
$f$; in general $H_f$ will be the full automorphism group of the point set $\Sigma$ since a general elliptic curve
does not have complex multiplication.\\
Let us see that $H_f$ is conjugate to $H$: $H_f$ is generated by two commuting reflections $\gamma_1$, $\gamma_2$
acting on the Riemann sphere $\mathbb{P}^1$ (with two fixed points each). By applying a suitable projectivity, we
see that $H_f$ is conjugate to $\langle \omega , \gamma_2' \rangle$ where $\gamma_2'$ is another reflection
\emph{commuting} with $\omega$; thus $\omega$ interchanges the fixed points of $\gamma_2'$ and also the fixed
points of $\rho$: Thus if we change coordinates via a suitable dilation (a projectivity preserving the fixed
points of $\omega$),
$\gamma_2'$ goes over to
$\rho$, and thus
$H_f$ is conjugate to $H$.\\
One computes that $\sigma$ and $\tau$ normalize $H$; in fact, $\sigma^{-1}\omega \sigma = \rho $, $\sigma^{-1}\rho
\sigma = \omega\rho $, and $\tau^{-1}\omega \tau =\omega \rho$, $\tau^{-1}\rho \tau = \rho$. Moreover, $\tau$ has
order
$4$ and $\sigma$ order $3$, $(\tau\sigma )^2=1$, thus one has the relations
\begin{gather*}
\tau^4 = \sigma^3 = (\tau \sigma )^2=1 \, .
\end{gather*}
It is known that $\mathfrak{S}_4$ is the group on generators $R$, $S$ with relations $R^4=S^2=(RS)^3=1$; mapping
$R\mapsto \tau^{-1}$, $S\mapsto \tau\sigma$, we see that the group $\langle \tau , \sigma\rangle < N(H)$ is a
quotient of $\mathfrak{S}_4$; since $\langle \tau , \sigma \rangle $ contains elements of order $4$ and order
$3$, its order is at least $12$, but since there are no normal subgroups of order $2$ in $\mathfrak{S}_4$,
$\mathfrak{S}_4=\langle \tau , \sigma \rangle$. To finish the proof, it therefore suffices to note that the order
of $N(H)$ is at most $24$: For this one just has to show that the centralizer of $H$ in
$\mathrm{PSL}_2\,\mathbb{C}$ is just $H$, for then $N(H)/H$ is a subgroup of the group of permutations of the
three nontrivial elements $H-\{ 1 \}$ in $H$ (in fact equal to it). Elements in $\mathrm{PGL}_2\,\mathbb{C}$
commuting with $\omega$ must be of the form
\begin{gather*}
\left[ \left( \begin{array}{cc} a & b \\ -b & a \end{array}\right)\right] \quad \mathrm{or}\quad 
\left[ \left( \begin{array}{cc} a & b \\ b & -a \end{array}\right)\right] \, ,
\end{gather*}
and if these commute also with $\rho$, the elements $1$, $\omega$, $\rho$, $\omega\rho$ are the only possibilities.
\end{proof}

\begin{corollary}
The variety $(V(4)^H)^0 \subset V(4)$ consisting of those points whose stabilizer in
$\mathrm{PSL}_2\,\mathbb{C}$ is exactly $H$ is a $(\mathrm{PSL}_2\,\mathbb{C} , N(H))$-section of $V(4)$.
\end{corollary}
\begin{proof}
The fact that the orbit $\mathrm{PSL}_2\,\mathbb{C} \cdot (V(4)^H)^0$ is dense in $V(4)$ follows since a general
point in $V(4)$ has stabilizer conjugate to $H$; the assertion $ \forall g\in\mathrm{PGL}_2\,\mathbb{C}$,
$\forall x\in (V(4)^H)^0$ : $gx \in (V(4)^H)^0 \implies g\in N(H)$ is clear by definition.
\end{proof}

Let us recall the representation theory of $N(H)=\mathfrak{S}_4$ viewed as the group of permutations of four
letters $\{ a,\: b,\: c,\: d\}$; the character table is as follows (cf. \cite{Se}).\\
\begin{center}
\begin{tabular}{c | rrrrr}
 & $1$ & $(ab)$ & $(ab)(cd)$ & $(abc)$ & $(abcd)$  \\ \hline
$\chi_0$       & $1\;$ & $1\;$ & $1\;$ & $1\;$ & $1\;$ \\
$\epsilon$     & $1\;$ & $-1\;$ & $1\;$ & $1\;$ & $-1\;$ \\
$\theta$       & $2\;$ & $0\;$ & $2\;$ & $-1\;$ & $0\;$ \\
$\psi$         & $3\;$ & $1\;$ & $-1\;$ & $0\;$ & $-1\;$ \\
$\epsilon\psi$ & $3\;$ & $-1\;$ & $-1\;$ & $0\;$ & $1\;$ 
\end{tabular}
\end{center}\vspace{0.3cm}
$V_{{\chi}_0}$ is the  trivial $1$-dimensional representation, $V_{\epsilon}$ is the $1$-dimensional
representation where $\epsilon (g)$ is the sign of the permutation $g$; $\mathfrak{S}_4=N(H)$ being the semidirect
product of $N(H)/H=\mathfrak{S}_3$ by the normal subgroup $H$, $V_{\theta}$ is the irreducible two-dimensional
representation induced from the representation of $\mathfrak{S}_3$ acting on the elements of $\mathbb{C}^3$ which
satisfy $x+y+z=0$ by permutation of coordinates. $V_{\psi}$ is the extension to $\mathbb{C}^3$ of the natural
representation of $\mathfrak{S}_4$ on $\mathbb{R}^3$ as the group of rigid motions stabilizing a regular
tetrahedron; finally, $V_{\epsilon\psi}=V_{\epsilon}\otimes V_{\psi}$.\\
We want to decompose $V(8)\oplus V(0)\oplus V(4)$ as $N(H)$-module; we fix the notation:
\begin{gather}
a_0:=1 ; \quad a_1 := z_1^4+z_2^4 , \; a_2:=6z_1^2z_2^2 ,\; a_3:= z_1^4-z_2^4 ,\\
 a_4:= 4(z_1^3z_2-z_1z_2^3) ,\; a_5:= 4(z_1^3z_2+z_1z_2^3) ; \nonumber\\
e_1:= 28(z_1^6z_2^2-z_1^2z_2^6) , \; e_2:= 56(z_1^7z_2 + z_1^5z_2^3- z_1^3z_2^5-z_1z_2^7) , \nonumber \\
e_3:= 56(z_1^7z_2 - z_1^5z_2^3- z_1^3z_2^5+z_1z_2^7) , \; e_4:= z_1^8-z_2^8 \nonumber \\
e_5:= 8(z_1^7z_2-7z_1^5z_2^3+ 7 z_1^3 z_2^5 -z_1 z_2^7) , \nonumber \\ e_6:=8(z_1^7z_2+7z_1^5z_2^3+ 7 z_1^3
z_2^5 +z_1 z_2^7) , \nonumber \\
e_7:= z_1^8+z_2^8 , \; e_8 := 28 (z_1^6z_2^2 + z_1^2z_2^6) , \; e_9:= 70 z_1^4z_2^4  \, . \nonumber
\end{gather}

\begin{lemma}
One has the following decompositions as $N(H)$-modules:
\begin{gather}
V(0)=V_{{\chi}_0} , \; V(4)=V_{\psi} \oplus V_{\theta} , \; V(8)= V_{\epsilon\psi}\oplus V_{\psi} \oplus
V_{\theta}\oplus V_{{\chi}_0}\, .
\end{gather}
More explicitly,
\begin{gather}
V(0) = \langle a_0 \rangle , \: V(4)= \langle a_3 , a_4 , a_5 \rangle \oplus \langle a_1 , a_2 \rangle , \\
V(8) = \langle e_4 , e_5 , e_6 \rangle \oplus \langle e_1 , e_2 , e_3 \rangle \oplus \langle e_8 , 7e_7-e_9
\rangle \oplus \langle 5e_7 + e_9 \rangle \, . \nonumber 
\end{gather}
Here $\langle e_4 , e_5 , e_6 \rangle$ corresponds to $V_{\epsilon\psi }$ and $\langle e_1 , e_2 , e_3 \rangle$
corresponds to $V_{\psi}$. Moreover, 
\begin{gather}
V(0)^H= \langle a_0 \rangle , \: V(4)^H = \langle a_1 , a_2 \rangle , \: V(8)^H = \langle e_7 , e_8 , e_9 \rangle
\, . 
\end{gather}
\end{lemma}

\begin{proof}
We will prove (25) first; one observes that quite generally for $k\ge 0$, $V(2k)^H= (V(2k)^{\rho })^{\omega}$
($\rho$ and $\omega$ commute) and that the monomials $z_1^j z_2^{2k-j}$, $j=0 , \dots , 2k$, are invariant under
$\rho$ if
$j+k$ is even, and otherwise anti-invariant, so if $k=2s$, $\dim V(2k)^{\rho} = 2s+1$, and if $k=2s+1$, $\dim
V(2k)^{\rho} = 2s+1$. Since $\omega$ is also a reflection, we have $2 \dim (V(2k)^{\rho})^{\omega} - \dim
V(2k)^{\rho} = \mathrm{tr}(\omega |_{V(2k)^{\rho}})$, and the trace is $1$ for $k=2s$, and $-1$ for $k=2s+1$, thus
\begin{gather*}
\dim V(2k)^H = s+1 , \; k= 2s , \quad \dim V(2k)^H = s , \; k=2s+1 \, .
\end{gather*}
In particular, the $H$-invariants in $V(0)$, $V(4)$, $V(8)$ have the dimensions as claimed in (25), and one checks
that the elements given there are indeed invariant.\\
To prove (23), we use the Clebsch-Gordan formula $V(2k)\otimes V(2) = V(2k+2)\oplus V(2k) \oplus V(2k-2)$ (cf.
(20)) iteratively together with the fact that the character of the tensor product of two representations of a
finite group is the product of the characters of each of the factors; since $V(2)$ has dimension $3$ and $\dim
V(2)^H=0$, $V(2)$ is irreducible; the value of the character of the $N(H)$-module $V(2)$ on $\tau$ is $1$, so
$V(2)=V_{\epsilon\psi }$. Now $V(2)\otimes V(2) = V(4) \oplus V(2) \oplus V(0)$, and looking at the character
table, one checks that 
\[
(\epsilon\psi)^2 = \chi_0 + (\epsilon\psi) + (\psi) + (\theta ) \, .
\]
This proves the decomposition in (23) for $V(4)$. The decomposition for $V(8)$ is proven similarly (one proves
$V(6)= V_{\psi}\oplus V_{\epsilon\psi} \oplus V_{\epsilon} $ first).\\
The proof of (24) now amounts to checking that the given spaces are invariant under $\sigma$ and $\tau$; finally
note that $V_{\epsilon\psi}$ corresponds to $\langle e_4 , e_5 , e_6 \rangle$ since the value of the character on
$\tau$ is $1$.
\end{proof}
Recall from Lemma 2.4.4 that we want to prove the rationality of $(\mathbb{P}\,
U_{\lambda})/\mathrm{PSL}_2\,\mathbb{C}$ and we can and will always assume in the sequel that $\lambda = (1,
6\epsilon , 1 , 6 )$ for $\epsilon \neq 0$. In view of Lemma 3.1.4 it will be convenient for subsequent
calculations to write the map
$\delta_{\lambda } : V(8) \oplus V(0) \oplus V(4) \to V(4)$ in terms of the basis $(e_1 , \dots , e_9 , a_0 , a_1
, \dots , a_5)$ in the source and the basis $(a_1, \dots , a_5)$ in the target. Denote coordinates in $V(8)\oplus
V(0) \oplus V(4)$ with respect to the chosen basis by $(x_1 , \dots, x_9 , s_0 , s_1 , \dots , s_5)=: (x, s)$.
Then one may write
\begin{gather}
\delta_{\lambda} (x, s) = \left( \begin{array}{c} Q_1(x, s) \\ \vdots \\ Q_5(x,s) \end{array}\right)
\end{gather}
with $Q_1(x,s) , \dots , Q_5(x,s)$ quadratic in $(x,s)$; their values may be computed using formulas (20), (22),
and the definition of $\delta_{\lambda}$ in Lemma 2.4.4, and they can be found in Appendix B. 

\begin{theorem}
Let $\tilde{\mathcal{Q}}_{\lambda} \subset V(8) \oplus V(0) \oplus V(4)$ be the subvariety defined by the
equations $Q_1=\dots = Q_5=0, \: s_3=s_4=s_5=0$. There is exactly one $7$-dimensional irreducible component
$\mathcal{Q}_{\lambda}$ of
$\tilde{\mathcal{Q}}_{\lambda}$ passing through the $N(H)$-invariant point $5e_7+e_9$ in $V(8)$;
$\mathcal{Q}_{\lambda}$ is $N(H)$-invariant and 
\begin{gather}
\mathbb{C}(\mathbb{P}\, U_{\lambda})^{\mathrm{PSL}_2\,\mathbb{C}} = \mathbb{C} (\mathbb{P}\,
\mathcal{Q}_{\lambda})^{N(H)}\, .
\end{gather}
\end{theorem}

\begin{proof}
We want to use Proposition 2.1.2, (2).\\
Note that $5e_7+e_9 \in U_{\lambda}$: In fact, $\delta_{\lambda}$ maps the $N(H)$-invariants in $V(8) \oplus V(0)
\oplus V(4)$ to the $N(H)$-invariants in $V(4)$ which are $0$. Since $U_{\lambda}$ is the unique irreducible
component of $\delta_{\lambda}^{-1}(0)$ passing through $a_0=1$, $U_{\lambda}$ contains the whole plane of
invariants $\langle a_0 , 5e_7 + e_9\rangle $.\\
If we denote by $p: V(8)\oplus V(0) \oplus V(4) \to V(4)$ the projection, then $\tilde{\mathcal{Q}}_{\lambda} =
p^{-1}(V(4)^{H}) \cap \delta_{\lambda}^{-1}(0)$. Clearly, $\tilde{\mathcal{Q}}_{\lambda}$ is $N(H)$-invariant, and
one only has to check that $5e_7+e_9$ is a nonsingular point on it with tangent space of dimension $7$ by direct
calculation: Then there is a unique $7$-dimensional irreducible component $\mathcal{Q}_{\lambda}$ of
$\tilde{\mathcal{Q}}_{\lambda}$ passing through $5e_7+e_9$ which is $N(H)$-invariant (since $5e_7+e_9$ is an
invariant point on it and this point is nonsingular on $\tilde{\mathcal{Q}}_{\lambda}$).\\
It remains to prove (27): $\mathcal{Q}_{\lambda}$ is an irreducible component of $p^{-1}(V(4)^H)\cap U_{\lambda}$
and $\mathcal{Q}_{\lambda}^0 =\mathcal{Q}_{\lambda}\cap p^{-1}((V(4)^H)^0)$ is a dense $N(H)$-invariant open subset
of $\mathcal{Q}_{\lambda}$ dominating $(V(4)^H)^0$. Thus by Proposition 2.1.1 (2),
\begin{gather*}
\mathbb{C}(\mathbb{P}\, U_{\lambda})^{\mathrm{PSL}_2\,\mathbb{C}} \simeq \mathbb{C} (\mathbb{P} \,
\mathcal{Q}_{\lambda}^0 )^{N(H)} \simeq \mathbb{C}( \mathbb{P}\,\mathcal{Q}_{\lambda} )^{N(H)} \, .
\end{gather*}
\end{proof}

\subsection{Dividing by the action of $H$}

Next we would like to "divide out" the action by $H$, so that we are left with an invariant theory problem for the
group $N(H)/H = \mathfrak{S}_3$. Look back at the action of $N(H)$ on $M:=\{ s_3=s_4=s_5 =0\} \subset V(8)\oplus
V(0)
\oplus V(4)$ which is explained in formulas (23), (24); we will adopt the notational convention to denote the
irreducible
$N(H)$-submodule of $V(8)$ isomorphic to $V_{\psi}$ by $V(8)_{(\psi )}$ and so forth; thus
\begin{gather}
M = V(0)_{( \chi_0 )} \oplus V(4)_{(\theta )} \oplus V(8)_{( \chi_0) } \oplus V(8)_{(\theta )} \oplus V(8)_{(\psi
)} \oplus V(8)_{(\epsilon\psi )}\, ,
\end{gather}
and looking at the character table of $\mathfrak{S}_4$, we see that the action of $H$ is nontrivial only on
$V(8)_{(\psi )} \oplus V(8)_{(\epsilon\psi )} = \langle e_1 , e_2 , e_3 \rangle \oplus \langle e_4 , e_5 , e_6
\rangle $ where $x_1, x_2, x_3$ and $x_4 , x_5 , x_6$ are coordinates; in terms of these, we have
\begin{gather}
(\omega )(x_1 , \dots , x_6) = ( -x_1 , \: x_2 , \: -x_3, \: -x_4, \: x_5 , \: -x_6) \, , \\
(\rho )(x_1 , \dots , x_6) = ( x_1 , \: -x_2 , \: -x_3, \: x_4, \: -x_5 , \: -x_6) \, , \nonumber \\
(\omega\rho )(x_1 , \dots , x_6) = ( -x_1 , \: -x_2 , \: x_3, \: -x_4, \: -x_5 , \: x_6)  \, ,\nonumber
\end{gather}
and 
\begin{gather}
\tau (x_1, \dots , x_6) = (-x_1 , \: -i x_3 , \: -i x_2 , \: x_4 , \: -ix_6 , \: -ix_5 )\, , \\
\sigma (x_1, \dots , x_6) =\left( 4x_3 , -\frac{i}{4} x_1 , \: ix_2 , \: -8x_6 ,\: -\frac{i}{8} x_4 , \: -ix_5
\right) \, . \nonumber
\end{gather}
Thus we see that the map
\begin{gather*}
\mathbb{P} ( V(8)_{(\psi )} \oplus V(8)_{( \epsilon\psi ) } ) - \{ x_1x_2x_3=0\}  \to R \times \mathbb{P}^2\, , \\
(x_1, \dots , x_6) \mapsto \left( \left( \frac{x_4}{x_1} , \frac{x_5}{x_2} , \frac{x_6}{x_3} \right) \: , \: \left(
\frac{1}{x_1^2} : \frac{1}{x_2^2} : \frac{1}{x_3^2} \right) \right) \, ,
\end{gather*}
where $R= \mathbb{C}^3$, is dominant with fibres $H$-orbits, and furthermore $N(H)$-equivariant for a suitable
action of $N(H)$ on $R\times \mathbb{P}^2$: In fact, we will agree to write 
\begin{gather*}
\left( \frac{1}{x_1^2} : \frac{1}{x_2^2} : \frac{1}{x_3^2} \right) = \left( \frac{x_2x_3}{x_1} :
\frac{x_3x_1}{x_2} : \frac{x_1x_2}{x_3} \right)
\end{gather*}
and remark that the subspaces 
\begin{gather*}
R= \left\langle \frac{x_4}{x_1} , \frac{x_5}{x_2} , \frac{x_6}{x_3} \right\rangle \, , \quad T:=  \left\langle 
\frac{x_2x_3}{x_1} , \frac{x_3x_1}{x_2} , \frac{x_1x_2}{x_3} \right\rangle
\end{gather*}
of the field of fractions of $\mathbb{C}\left[ V(8)_{(\psi)} \oplus V(8)_{(\epsilon\psi )} \right]$ are invariant
under $\sigma$ and $\tau$ (thus $\mathbb{P}^2 = \mathbb{P} (T) $). If we denote the coordinates with respect to
the basis vectors in $R$ resp. $T$ given above by $r_1, r_2 , r_3$ resp. $y_1, y_2, y_3$, then the actions of
$\tau$ and $\sigma$ are described by
\begin{gather*}
\tau (r_1 , r_2 , r_3) = (-r_1 , r_3 , r_2) \, , \; \sigma (r_1 , r_2 , r_3) = (-2r_3 , r_1/2 , -r_2) \, \\
\tau (y_1 , y_2 , y_3) = (y_1 , -y_3 , -y_2) \, , \; \sigma (y_1, y_2 , y_3) = ((1/16) y_3 , -16y_1 , -y_2)\, .
\end{gather*}
Thus the only $N(H)$-invariant lines in $R$ resp. $T$ are the ones spanned by $(2 , 1 , -1)$ resp. $(-1, 16, -16)$
on which $\tau$ acts by multiplication by $-1$ resp. by $+1$ and hence
\begin{gather}
R= R_{(\epsilon )} \oplus R_{(\theta )} \, , \; T = T_{(\chi_0 )} \oplus T_{(\theta )} \, .
\end{gather}
We see that the morphism
\begin{gather}
\pi \: : \: \mathbb{P} ( M ) - \{ x_1x_2x_3=0\}  \\ \to  R \times
\mathbb{P} ( T\oplus
 V(8)_{( \chi_0) } \oplus V(8)_{(\theta )} \oplus V(0)_{(\chi_0 )} \oplus V(4)_{(\theta )} ) \simeq
R\times \mathbb{P}^8 \, , \nonumber \\
\pi (x,s) := \left( \left( \frac{x_4}{x_1} , \frac{x_5}{x_2} , \frac{x_6}{x_3} \right) \: , \: \left(
\frac{x_2x_3}{x_1} : \frac{x_3x_1}{x_2} : \frac{x_1x_2}{x_3} \right) \right. \nonumber \\
\left. : x_7 : x_8 : x_9 : s_0 :
s_1 : s_2 \nonumber
\right)
\end{gather}
is $N(H)$-equivariant, dominant, and all fibres are $H$-orbits. If we consider $(x_7 , x_8 , x_9, s_0, s_1 , s_2)$
as coordinates in $V(8)_{( \chi_0) } \oplus V(8)_{(\theta )} \oplus V(0)_{(\chi_0 )} \oplus V(4)_{(\theta )}$ in
the \emph{target} of the map $\pi$ (as we do in formula (32)) we denote them by $(y_7, y_8, y_9, y_{10}, y_{11},
y_{12})$ to achieve consistency with \cite{Kat2}.\\
How do we get equations which define the image
\begin{gather*}
\pi (\mathbb{P}\, \tilde{\mathcal{Q}}_{\lambda } \cap \{x_1x_2 x_3 \neq 0 \} ) \subset R\times (\mathbb{P}^8 -
\{y_1y_2y_3 =0 \} )
\end{gather*}
in $\mathbb{P}^8 - \{y_1y_2y_3 =0 \}$ from the quadrics $Q_1(x,s), \dots , Q_5(x,s)$ in formula (26)? We can set
$s_3=s_4=s_5=0$ in $Q_1,\dots , Q_5$ to obtain equations $\bar{Q}_1, \dots , \bar{Q}_5$ for
$\mathbb{P}\, \tilde{\mathcal{Q}}_{\lambda }$ in $\mathbb{P} (M)$; the point is now that the quantities 
\begin{gather*}
\bar{Q}_1, \; \bar{Q}_2 , \; \frac{\bar{Q}_3}{x_1} , \; \frac{\bar{Q}_4}{x_2} , \; \frac{\bar{Q}_4}{x_3}
\end{gather*}
are $H$-invariant (as one sees from the equations in Appendix B). Moreover, the map 
\begin{gather*}
\pi : \mathbb{P} (M) - \{x_1x_2x_3 = 0 \} \to R\times (\mathbb{P}^8 -\{ y_1y_2y_3=0\})
\end{gather*}
is a geometric quotient for the action of $H$ on the source (by \cite{Po-Vi}, Thm. 4.2), so we can write
\begin{gather*}
\bar{Q}_1 =q_1 (r_1, \dots , y_{12}) , \; \bar{Q}_2 = q_2 (r_1, \dots , y_{12}) , \; \frac{\bar{Q}_3}{x_1} = q_3
(r_1, \dots , y_{12}), \\
\frac{\bar{Q}_4}{x_2} = q_4 (r_1, \dots , y_{12}) , \;
\frac{\bar{Q}_4}{x_3}= q_5 (r_1, \dots , y_{12})
\end{gather*} 
where $q_1, \dots , q_5$ are polynomials in $(r_1, r_2 , r_3)$, $(y_1, y_2, y_3, y_7, \dots , y_{12})$ which one
may find written out in Appendix B. Here we just want to emphasize their structural properties which will
be most important for the subsequent arguments:\\
\begin{itemize}
\item[(1)]
The polynomials $q_1$, $q_2$ are homogeneous of degree $2$ in the set of variables $(y_1, \dots , y_{12})$;
 the coefficients of the monomials in the $y$'s are (inhomogeneous) polynomials of degrees $\le 2$ in $r_1, \:
r_2, \: r_3$. For $r_1=r_2=r_3=0$, $q_1, q_2$ do not vanish identically. 
\item[(2)]
The polynomials $q_3$, $q_4$, $q_5$ are homogeneous linear in $(y_1, \dots , y_{12})$; the coefficients of the
 monomials in the $y$'s are (inhomogeneous) polynomials of degrees $\le 2$ in $r_1, \:
r_2, \: r_3$. For $r_1=r_2=r_3=0$, $q_3, q_4, q_5$ do not vanish identically. 
\end{itemize}

\begin{theorem}
Let $\tilde{Y}_{\lambda}$ be the subvariety of $R\times\mathbb{P}^8$ defined by the equations
$q_1=q_2=q_3=q_4=q_5=0$. There is an irreducible $N(H)$-invariant component $Y_{\lambda }$ of
$\tilde{Y}_{\lambda}$ with $\pi ([x^0 ]) \in Y_{\lambda}$, where $x^0 : = 13i (5e_7 + e_9) + 5 (4e_1-ie_2+e_3)$,
such that
\begin{gather}
\mathbb{C} (\mathbb{P}\, \mathcal{Q}_{\lambda})^{N(H)} \simeq \mathbb{C} (Y_{\lambda})^{N(H)} \, .
\end{gather}
\end{theorem}

\begin{proof}
The variety $Y_{\lambda}$ will be the closure of the image $\pi (\mathbb{P}\,\mathcal{Q}_{\lambda} \cap
\{x_1x_2x_3\neq 0\} )$ in $R\times\mathbb{P}^8$.\\
It remains to see that $x^0 \in \mathcal{Q}_{\lambda}$. Recall from Theorem 3.1.5 that $\mathcal{Q}_{\lambda}$ is
the unique irreducible component of $\tilde{\mathcal{Q}}_{\lambda}$ passing through the $N(H)$-invariant point
$5e_7+e_9$, and that this point is a nonsingular point on $\tilde{\mathcal{Q}}_{\lambda}$; thus, if we can find an
irreducible subvariety of $\tilde{\mathcal{Q}}_{\lambda}$ which contains both $5e_7+e_9$ and $x^0$, we are done.
The sought-for subvariety is $\tilde{\mathcal{Q}}_{\lambda} \cap V(8)^{\sigma }$, where $V(8)^{\sigma}$ are the
elements in $V(8)$ invariant under $\sigma\in N(H)$. One sees that $x^0$ and $5e_7+e_9$ lie on it, and computing 
\begin{gather*}
V(8)^{\sigma}= \langle 5e_7 + e_9 , 8e_4 -ie_5 -e_6 , 4e_1  -ie_2 +e_3 \rangle \, , \\
V(4)^{\sigma}= \langle 2(z_1^4-z_2^4)+ 4 (z_1^3z_2 +z_1z_2^3) +4i (z_1^3z_2 - z_1z_2^3) \rangle  \, ,
\end{gather*}
and using $\delta_{\lambda} (V(8)^{\sigma}) \subset V(4)^{\sigma}$, we find that
$\tilde{\mathcal{Q}}_{\lambda}\cap V(8)^{\sigma}$ is a quadric in $V(8)^{\sigma}$ which is easily checked to be
irreducible. 
\end{proof}

Thus it remains to prove the rationality of $Y_{\lambda}/N(H)=Y_{\lambda}/\mathfrak{S}_3$. 

\subsection{Inner projections and the "no-name" method}
The variety $\tilde{Y}_{\lambda}$ comes with the two projections
\begin{gather*}
\begin{CD}
\tilde{Y}_{\lambda} @>{p_{\mathbb{P}^8}}>> \mathbb{P}^8 \\
@V{p_R}VV \\
R
\end{CD}
\end{gather*}
Recall from (32) that $N:= \mathbb{P} ( V(8)_{\theta} \oplus V(4)_{\theta}) \subset \mathbb{P}^8$ is an
$N(H)$-invariant $3$-dimensional projective subspace of $\mathbb{P}^8$. We will show
$\mathbb{C}(Y_{\lambda})^{N(H)} \simeq \mathbb{C}(R\times N)^{N(H)}$ via the following theorem.

\begin{theorem}
There is an open $N(H)$-invariant subset $R_0\subset R$ containing $0\in R$ with the following properties:
\begin{itemize}
\item[(1)]
For all $r\in R_0$ the fibre $p_R^{-1}(r) \subset \tilde{Y}_{\lambda}$ is irreducible of dimension $3$, and
$p_R^{-1}(R_0)$ is an open $N(H)$-invariant subset of $Y_{\lambda}$.
\item[(2)]
There exist $N(H)$-sections $\sigma_1$, $\sigma_2$ of the $N(H)$-equivariant projection $R_0\times \mathbb{P}^8 \to
R_0$ such that $N(r) : = \langle \sigma_1(r) , \: \sigma_2 (r) , (1: 0 : 0 : \dots : 0), (0: 1: 0: \dots : 0), (0:
0: 1 : 0 : \dots : 0) \rangle \subset \mathbb{P}^8$, $r\in R_0$, is an $N(H)$-invariant family of $4$-dimensional
projective subspaces in $\mathbb{P}^8$ with the properties:
\begin{itemize}
\item[(i)]
$N(r)$ is disjoint from $N$ for all $r\in R_0$.
\item[(ii)]
The fibre $p_{\mathbb{P}^8}(p_R^{-1}(r)) \subset \mathbb{P}^8$ contains the line $\langle \sigma_1(r) ,
\sigma_2(r)\rangle \subset N(r)$ for all $r\in R_0$.
\item[(iii)]
The projection $\pi_r : \mathbb{P}^8 \dasharrow N$ from $N(r)$ to $N$ maps the fibre
$p_{\mathbb{P}^8}(p_R^{-1}(r)) \subset \mathbb{P}^8$ dominantly onto $N$ for all $r\in R_0$.
\end{itemize}
\end{itemize}
\end{theorem}

Before turning to the proof, let us note the following corollary.

\begin{corollary}
One has the field isomorphism 
\begin{gather*}
\mathbb{C} (Y_{\lambda})^{N(H)} \simeq \mathbb{C} (R\times N)^{N(H)}\, ,
\end{gather*}
and the latter field is rational. Hence $\mathfrak{M}_3$ is rational. 
\end{corollary}

\begin{proof}
(of corollary) The $N(H)$-invariant set $p_R^{-1}(R_0)$ is an open subset of $Y_{\lambda}$. Let us see that the
projection $\pi_r :  F_r:= p_{\mathbb{P}^8} (p_R^{-1}(r)) \dasharrow N$ is birational. In fact, $F_r$ is of
dimension $3$ and irreducible and the intersection of a $3$-codimensional linear subspace and two quadrics in
$\mathbb{P}^8$. Moreover, $F_r\cap N(r)$ contains a line $L_r$ by Theorem 3.3.1 (2), (ii). Thus for a general
point $P$ in $N$, $F_r \cap \langle L_r , P\rangle $ consists of $L_r$ and a single point (namely the point of
intersection of the two lines which are the residual intersections of each of the two quadrics defining $F_r$ with
$\langle L_r , P\rangle $, the other component being $L_r$ itself). Thus $\pi_r$ is generically one-to-one whence
birational.\\
Thus one has a birational $N(H)$-isomorphism $p_R^{-1}(R_0) \dasharrow R_0 \times N$, given by sending $(r, [y])$
to $(r , \pi_r ([y]))$. Thus one gets the field isomorphism in Corollary 3.3.2.\\
By the no-name lemma (cf. e.g. \cite{Dol1}, section 4), $\mathbb{C}(R\times N)^{N(H)} \simeq \mathbb{C} (N)^{N(H)}
(T_1, T_2, T_3)$, where $T_1, \: T_2, \: T_3$ are indeterminates, thus it suffices to show that the quotient of
$N$ by $N(H)$ is stably rational of level $\le 3$. This in turn follows from the same lemma, since clearly, if we
take the representation of $\mathfrak{S}_3$ in $\mathbb{C}^3$ by permutation of coordinates, the quotient of
$\mathbb{P} (\mathbb{C}^3 )$ by $\mathfrak{S}_3$, a unirational surface, is rational.
\end{proof}

\begin{proof}
(of theorem) The proof will be given in several steps.\\
\emph{Step 1. (Irreducibility of the fibre over $0$)} We have to show that the variety
$p_{\mathbb{P}^8 } p_R^{-1}(0)) \subset \mathbb{P}^8$ is irreducible and $3$-dimensional. We have explicit
equations for it (namely the ones that arise if we substitute $r_1=r_2=r_3=0$ in $q_1 , \dots , q_5$, which are
thus $3$ linear and $2$ quadratic equations); the assertions can then be checked with a computer algebra system
such as \ttfamily{Macaulay 2}\rmfamily . Recall from Theorem 3.2.1 that $Y_{\lambda}$ contains $\pi ([x^0] )$. In
fact,
\begin{gather}
\pi ([x^0]) = \left( (0,0,0) , \left( -\frac{5}{4} : 20 : -20 : 65 : 0 : 13 : 0 : 0 : 0 \right)\right) \, ,
\end{gather}
as follows from the definition of $x^0$ in Theorem 3.2.1 and the definition of $\pi$ in (32). Thus $\pi ([x^0])$
lies in the fibre over $0$ of $p_R^{-1}$ and thus, since there is an open subset around $0$ in $R$ over which the
fibres are irreducible and $3$-dimensional, assertion (1) of Theorem 3.3.1 is established.\\
\emph{Step 2. (Construction of $\sigma_1$)} To obtain $\sigma_1$, we just assign to $r\in R$ the point $(r,
\sigma_1(r))$ with $\sigma_1(r) = (0: 0: 0: 0 : 0 : 0 : 1 : 0: 0)$, i.e. $y_{10}=1$, the other $y$'s being $0$.
This always is in the fibre $p_{\mathbb{P}^8}(p_R^{-1}(r))$ as one sees on substituting in the equations
$q_1,\dots , q_5$. Moreover, this is an $N(H)$-section, since $y_{10}$ is a coordinate in the space $V(0)_{\chi_0}$
in formula (32).\\
\emph{Step 3. (Construction of $\sigma_2$; decomposition of $V:=\mathbb{P}(\delta_{\lambda}^{-1}(0)\cap V(8))$ )}
The construction of a section $\sigma_2$, $\sigma_2 (r) = (\sigma_2 ^{(1)}(r) : \dots : \sigma_2^{(9)}(r) )$,
involves a little more work. Let us look back at the construction of $Y_{\lambda}$ in subsection 3.2 for this,
especially the definition of the projection $\pi$ in formula (32), and the decomposition of the linear subspace
$M\subset V(8)\oplus V(0) \oplus V(4)$. By definition of $R$, the family of codimension $3$ linear subspaces
\begin{gather}
L(r) : = \{ [(x,s)] \, | \, x_4=r_1x_1 , \: x_5= r_2x_2 , \: x_6= r_3x_3 \} \subset \mathbb{P} (M)\, ,
\end{gather}
$r=(r_1 , r_2, r_3)\in R$, is $N(H)$-invariant, i.e. $g L(r)= L(gr)$, for $g\in N(H)$. It is natural to intersect
this family with $\mathbb{P} (\delta_{\lambda}^{-1}(0) \cap V(8))$ which, as we will see, has dimension $3$ and
look for an $H$-orbit $\mathfrak{O}_r$ in the intersection of $\mathbb{P} (\delta_{\lambda}^{-1}(0)\cap V(8))$ with the open
set of
$L(r)$ where $x_1x_2x_3\neq 0$. Moreover, we will see that for $r=0$, the point $[x^0]$ is in this intersection.
Thus passing to the quotient we may put
\begin{gather}
(r , \sigma_2(r)):= \pi ( \mathfrak{O}_r )
\end{gather}
to obtain a $\sigma_2$ with the required properties. Indeed, note that we will have $\sigma_2^{(7)}(r)= \sigma_2^{(8)}
= \sigma_2^{(9)} =0$ which ensures that $\sigma_2$ and $\sigma_1$ span a line. Moreover,
\begin{gather}
\sigma_2 (0) = \left( -\frac{5}{4} : 20 : -20 : 65 : 0 : 13 : 0 : 0 : 0 \right) \, ,
\end{gather}
by formula (34), which allows us to check assertions (2), (i) and (iii) of Theorem 3.3.1, which are open
properties on the base $R$, by explicit computation for the fibre over $0$. Property (2), (ii) stated in the
theorem is clear by construction. Let us now carry out this program. We will start by explicitly decomposing
$V:=\mathbb{P}(\delta_{\lambda}^{-1}(0)\cap V(8))$ into irreducible components.\\
To guess what $V$ might be, note that according to the definition of $\delta_{\lambda}$ in Lemma 2.4.4,
$\delta_{\lambda}$ vanishes on $f_8\in V(8)$ if for the transvectant $\psi_6$ one has $\psi_6(f_8,f_8)=0$; but
looking back at the definition of transvectants in formula (20), we see that $\psi_6 : V(8) \times V(8) \to V(4)$
vanishes if $f_8$ is a linear combination of $z_1^8$, $z_1^7z_2$ and $z_1^6z_2^2$ (since we differentiate at least
$3$ times with respect to $z_2$ in one factor in the summands in formula (20)). Thus $X_1:=
\overline{\mathrm{PSL}_2\, \mathbb{C}\cdot \langle z_1^8,\: z_1^7z_2,\: z_1^6z_2^2 \rangle }$, the variety of
forms of degree $8$ with a six-fold zero, is contained in $V$, and one computes that the differential of
$\delta_{\lambda}|_{V(8)}$ in $z_1^6z_2^2$ is surjective, so that $X_1$ is an irreducible component of $V$.\\
The dimension of $X_1$ is clearly three. Weyman, in \cite{Wey}, Cor. 4, computed the Hilbert function of $X_{p,g}$,
the variety of binary forms of degree $g$ having a root of multiplicity $\ge p$ which is
\begin{gather*}
H(X_{p,g}, d) = (dp+1){ g-p+d \choose g-p} - (d(p+1) -1) {g-p+d-1 \choose g-p-1} \, .
\end{gather*}
For $d=6$, $g=8$, the leading term in $d$ in this expression is $3d^3$, which shows
\begin{gather}
\deg X_1= 18 \, .
\end{gather}
Moreover, we know already that $5e_7+e_9$ is in $V$ from the proof of Theorem 3.1.5; thus set $X_2:=
\overline{\mathrm{PSL}_2\, \mathbb{C} \cdot \langle 5e_7+e_9 \rangle}$. We know that the stabilizer of $5e_7
+e_9$ in $\mathrm{PSL}_2\,\mathbb{C}$ contains $N(H)$ because $5e_7+e_9 =5z_1^8+5z_2^8+70z_1^4z_2^4$ spans the
$N(H)$-invariants in $V(8)$ by Lemma 3.1.4. The claim is that the stabilizer is not larger. An easy way to
check this is to use the beautiful theory developed in \cite{Ol}, p. 188 ff., using differential invariants and
signature curves, which allows the explicit determination of the order of the symmetry group of a complex binary
form. More precisely we have (cf. \cite{Ol}, Cor. 8.68):
\begin{theorem}
Let $Q(p)$ be a binary form of degree $n$ (written in terms of the inhomogeneous coordinate $p=z_1/z_2$) which is
not equivalent to a monomial. Then the cardinality $k$ of the symmetry group of $Q(p)$ satisfies
\begin{gather*}
k \le 4n -8 \, ,
\end{gather*}
provided that $U$ is not a constant multiple of $H^2$, where $U$ and $H$ are the following polynomials in $p$: $H
:= (1/2) (Q,Q)^{(2)}$, $T:= (Q, H)^{(1)}$, $U:= (Q, T)^{(1)}$ where, if $Q_1$ is a binary form of degree $n_1$,
and $Q_2$ is a binary form of degree $n_2$, we put
\begin{gather*}
(Q_1 , Q_2)^{(1)} : = n_2 Q_1'Q_2 - n_1 Q_1Q_2' \, , \\
(Q_1 , Q_2)^{(2)} : = n_2(n_2-1) Q_1'' Q_2 - 2 (n_2-1)(n_1-1) Q_1'Q_2' \\
+ n_1 (n_1-1) Q_1 Q_2'' \, .
\end{gather*}
(these are certain transvectants). 
\end{theorem}
Applying this result in our case, we find the upper bound $24$ for the symmetry group of $5e_7+e_9$, which is
indeed the order of $N(H)= \mathfrak{S}_4$. $X_2$ is irreducible of dimension $3$, and computing that the
differential of $\delta_{\lambda}|_{V(8)}$ is surjective in $5e_7+e_9$, we get that $X_2$ is another irreducible
component of $V$. But let us intersect $X_2$ with the codimension $3$ linear subspace in $V(8)$ consisting of
forms with zeroes $\zeta_1, \zeta_2, \zeta_3 \in\mathbb{P}^1$; there is a unique projectivity carrying these to
three roots of $5e_7+e_9$, which are all distinct, thus there are $8\cdot 7\cdot 6$ such projectivities, and
$\deg X_2 \ge (8\cdot 7\cdot 6)/ | N(H)| $. But one checks easily that $V$ itself has dimension $3$ and is the
intersection of $5$ quadrics in $\mathbb{P} (V(8))$, thus has degree $\le 32$. Thus we must have
\begin{gather}
\deg X_2 = 14, \: V = X_1 \cup X_2 , \: \deg V =32 . 
\end{gather}
Note also that 
\begin{gather}
[x^0] \in X_2 \cap L(0)\, .
\end{gather}
In fact, from the proof of Theorem 3.2.1, we know $[x^0] \in V$, and $[x^0] \in L(0)$ being clear, we just check 
that $x^0$ has no root of multiplicity $\ge 6$.\\
\emph{Step 4. (Construction of $\sigma_2$; intersecting $V$ with a family of linear spaces in $\mathbb{P}(M)$)}
Let $L^0 (r)$ be the open subset of $L(r)\subset \mathbb{P}(M)$ where $x_1x_2x_3 \neq 0$. According to the strategy
outlined at the beginning of Step 3, we would like to compute the cardinalities 
\begin{gather*}
| L^0 (r) \cap X_1 | , \quad | L^0(r) \cap X_2 | , 
\end{gather*}
for $r$ varying in a small neighbourhood of $0$ in $R$. It is, however, easier from a computational point of view 
to determine the number of intersection points of $X_1$ resp. $X_2$ with certain boundary components of $L^0(r)$ in
$L(r)$ first; the preceding cardinalities will afterwards fall out as the residual quantities needed to have $\deg
X_1=18$, $\deg X_2 = 14$. Thus let us introduce the following additional strata of $L(r)\backslash L^0(r)$: 
\begin{gather}
 L_0:= \{ [(x,s)] \, | \, x_1=x_2=x_3=x_4=x_5=x_6=0 \} , \\
 L_1 (r):= \{ [(x,s)] \, | \, x_1 \neq 0 ,\: x_4= r_1x_1 , \: x_2=x_3=x_5=x_6=0 \} , \nonumber\\
 L_2(r):= \{ [(x,s)] \, | \, x_2 \neq 0 ,\: x_5= r_2x_2 , \: x_1=x_3=x_4=x_6=0 \} , \nonumber\\
 L_3(r):= \{ [(x,s)] \, | \, x_3 \neq 0 ,\: x_6= r_3x_3 , \: x_1=x_2=x_4=x_5=0 \} , \nonumber\\
 \tilde{L}_1 (r):= \{ [(x,s)] \, | \, x_2x_3 \neq 0, \: x_5= r_2x_2, \: x_6=r_3x_3, \: x_1=x_4=0 \} \nonumber\\
\tilde{L}_2 (r):= \{ [(x,s)] \, | \, x_1x_3 \neq 0, \: x_4= r_1x_1, \: x_6=r_3x_3, \: x_2=x_5=0 \} \nonumber\\
\tilde{L}_3 (r):= \{ [(x,s)] \, | \, x_1x_2 \neq 0, \: x_4= r_1x_1, \: x_5=r_2x_2, \: x_3=x_6=0 \} \nonumber .
\end{gather}
$L(r)$ is the disjoint union of these and $L^0 (r)$. From the equations describing $\delta_{\lambda}$ one sees
that $V$ is defined in $\mathbb{P}(V(8))$ with coordinates $x_1,\dots , x_9$ by
\begin{gather}
-192 x_6^2 -192 x_3x_6 +384 x_3^2 -192 x_5^2 -192 x_2x_5 +384 x_2^2 \\
- 12 x_1x_4 + 12 x_7x_8+180 x_8x_9 =0 ,\nonumber\\
64 x_6^2 -192 x_3x_6 -128 x_3^2 -64 x_5^2 +192 x_2x_5 +128 x_2^2 \\
-2x_4^2 +16 x_1^2 + 2x_7^2 - 16 x_8^2 -50x_9^2 = 0, \nonumber \\
96 x_5x_6 -672 x_3x_5 -672 x_2x_6 +1248 x_2x_3 \\
-12 x_1x_7 +12 x_4x_8 + 180 x_1x_9=0 , \nonumber \\
6x_4x_6 +42 x_3x_4 +84 x_1x_6 +156 x_1x_3 \\
-6x_5x_7 -42 x_2x_7 + 24 x_5x_8 - 264 x_2x_8 + 30 x_5x_9 -30 x_2x_9=0 , \nonumber\\
-6x_4x_5 -42 x_2x_4 + 84 x_1x_5 + 156 x_1x_2 \\
+ 6x_6x_7 + 42 x_3x_7 + 24 x_6x_8 - 264 x_3x_8 - 20 x_6x_9 + 30 x_3x_9 =0 \nonumber , 
\end{gather}
and thus
\begin{gather}
\tilde{L}_i (r) \cap V = \emptyset \quad \forall i=1,2,3
\end{gather}
for $r$ in a Zariski open neighbourhood of $0\in R$ (for $\tilde{L}_1 (r)$ consider equation (44) and assume
$96r_2r_3 - 672 r_2 - 672 r_3 +1248 \neq 0$, for $\tilde{L}_2 (r)$ we see that (45) cannot hold if $6r_1r_3 + 42
r_1 + 84 r_3 + 156 \neq 0$, and for $\tilde{L}_3 (r)$ equation (46) is impossible provided that $-6r_1r_2-42 r_1+
84 r_2 + 156 \neq 0$).\\
Let us consider the intersection $V \cap L_0$. We have to solve the equations
\begin{gather*}
12 x_7x_8 + 180 x_8x_9 =0 , \quad 2x_7^2 - 16 x_8^2 - 50 x_9^2 =0 , 
\end{gather*}
which have the four distinct solutions $(x_7,x_8,x_9) = (5 , 0 , \pm 1)$, $(x_7,x_8,x_9) = (15, \pm 5 , -1)$,
whence
\begin{gather}
L_0 \cap V = \{ [ 5e_7 \pm e_9] , [15 e_7 \pm 5 e_8 - e_9]  \} \, .
\end{gather}
We will also have to determine the intersection $V \cap L_1 (r)$ explicitly. We have to solve the equations
\begin{gather*}
-12 r_1 x_1^2 +12 x_7x_8 + 180 x_8x_9 =0, \\
-2r_1^2x_1^2 + 16 x_1^2 + 2x_7^2 - 16 x_8^2 - 50 x_9^2 = 0 , \\
-12 x_1x_7 + 12 r_1 x_1 x_8 + 180 x_1x_9 =0 ,
\end{gather*}
in the variables $x_1, x_7, x_8, x_9$. We can check (e.g. with \ttfamily Macaulay 2\rmfamily ) that the subscheme
they define has dimension $0$ (and degree $8$) for $r_1=0$. We already know four solutions with $x_1=0$, namely the
ones given in formula (48). Then it suffices to check that 
\begin{gather*}
(x_1, x_7, x_8, x_9) = (\pm 1, r_1, 1, 0) , \: (x_1, x_7, x_8, x_9) = (\pm a , (90-5r_1^2) , -5r_1, 6) , 
\end{gather*}
where $a$ is a square-root of $25 (r_1^2-36)$, are also solutions (with $x_1\neq 0$ in a neighbourhood of $0$ in
$R$, and obviously all distinct there). Thus
\begin{gather}
L_1(r) \cap V = 
\{ [\pm (e_1+r_1e_4) + r_1 e_7 + e_8] ,\\
 \: [\pm(ae_1 +r_1a e_4) + (90-5r_1^2)e_7 - 5r_1e_8 + 6e_9 ] \} \, . \nonumber
\end{gather}
We still have to see how the intersection points $L_0 \cap V$ and $L_1 (r) \cap V$ are distributed among $X_1$ and
$X_2$: Suppose $f\in V(8)$ is a binary octic such that $[f] \in L_0 \cap \mathbb{P}(V(8))$ or $[f] \in L_1 (r) \cap
\mathbb{P} (V(8))$; then $f$ is a linear combination of the binary octics $e_1$, $e_4$, $e_7$, $e_8$, $e_9$
defined in (22), which involve only even powers of $z_1$ and $z_2$; thus if $(a:b)\in\mathbb{P}^1$ is a root of
one of them, so is its negative $(a:-b)$ whence
\begin{quotation}
$[f]$ lies in $X_1$ if and only if $(1:0)$ or $(0:1)$ is a root of multiplicity $\ge 6$.
\end{quotation} 
Applying this criterion, we get, using (48) and (49)
\begin{gather}
L_0 \cap X_1 = \emptyset , \; L_0 \cap X_2 = \{ [ 5e_7 \pm e_9] , [15 e_7 \pm 5 e_8 - e_9]  \} , \\
L_1(r) \cap X_1= \{ [\pm (e_1+r_1e_4) + r_1 e_7 + e_8] \} , \nonumber \\
\; L_1(r)\cap X_2 = \{ [\pm(ae_1 +r_1a e_4)
+ (90-5r_1^2)e_7 - 5r_1e_8 + 6e_9 ]\} \, . \nonumber
\end{gather}
The reader may be glad to hear now that we do not have to repeat this entire procedure for
$L_2(r)$ and
$L_3(r)$; in fact, $L_1(r)$, $L_2(r)$, $L_3(r)$ are permuted by $N(H)$ in the following way: For the element 
$\sigma\in N(H)$ we have
\begin{gather*}
\sigma \cdot L_1(r) = L_2 (\sigma \cdot r) , \quad \sigma \cdot L_2(r) = L_3 (\sigma \cdot r), \quad \sigma \cdot
L_3(r) = L_1 (\sigma \cdot r)\, ,
\end{gather*}
which follows from (30) (and (28)) and the definition of $R$. 
 Thus we get that
generally for $i=1,2,3$
\begin{gather}
L_i(r) \cap X_1 = \{ P_1(r) , \: P_2(r) \} , \: L_i(r) \cap X_2 = \{ Q_1 (r) ,\: Q_2(r) \} 
\end{gather}
where $P_1(r)$, $P_2(r)$, $Q_1(r)$, $Q_2(r)$ are mutually distinct points, and this is valid in a Zariski open
$N(H)$-invariant neighbourhood of $0\in R$. It remains to check that
\begin{quotation}
$L(0) \cap V$ consists of $32$ reduced points.
\end{quotation}
We check (with \ttfamily Macaulay 2\rmfamily ) that if we substitute $x_4=x_5=x_6=0$ in equations (42)-(46), they
define a zero-dimensional reduced subscheme of degree $32$ in the projective space with coordinates $x_1, x_2, x_3,
x_7, x_8, x_9$. Taking into account (47), (50), (51), we see that all the intersections in equations (50), (51)
are free of multiplicities in an open $N(H)$-invariant neighbourhood of $0\in R$ and moreover, since $\deg
X_1=18$, $\deg X_2=14$, we must have there
\begin{quotation}
$L^0(r)\cap X_1$ consists of $12$ reduced points, and $L^0(r)\cap X_2$ consists of $4$ reduced points.
\end{quotation}
Now these $4$ points make up the $H$-orbit $\mathfrak{O}_r$ we wanted to find in Step 3: Clearly $L^0(r)\cap X_2$
is $H$-invariant, and $H$ acts with trivial stabilizers in $L^0(r)$ (as is clear from (29)). Thus we have completed
the program outlined at the beginning of Step 3. It just remains to notice that $[x^0]\in X_2 \cap L^0(0)$. This is
clear since
$[x^0]\in V$, but $x^0$ does not have a root of multiplicity $\ge 6$.
\\
\emph{Step 5. (Verification of the properties of $N(r)$)} For the completion of the proof of Theorem 3.3.1, it
remains to verify the properties of the subspace $N(r)$ in parts (2), (i) and (iii) of that theorem. First of all,
it is clear that
\begin{gather*}
N(r) = \langle \sigma_1( r), \sigma_2 (r), \: (1:0:0:\dots : 0),\\
\: (0:1:0: \dots :0), \: (0:0:1:\dots : 0)
\rangle
\end{gather*}
is $N(H)$-invariant in the sense that $g\cdot N(r) = N( g\cdot r)$ for $g\in N(H)$ 
by the construction of $\sigma_1$, $\sigma_2$ and because the last three vectors in the preceding formula are a
basis in the invariant subspace $\mathbb{P} (T) \subset \mathbb{P}^8$ (where by (31) $T=T_{(\chi_0)} \oplus
T_{(\theta )}$). Moreover, by the definition of $\sigma_1$ in Step 2, and the formula (37) for $\sigma_2(0)$, one
has $\dim N(0) = 4$, which thus holds also for $r\in R$ sufficiently close to $0$.\\
Recall that $N$ was defined to be $N:= \mathbb{P} (V(8)_{(\theta)} \oplus V(4)_{(\theta)}) \subset \mathbb{P}^8$,
and as such can be described in terms of the coordinates $(y_1 : y_2 : y_3 : y_7 : y_8 : \dots : y_{12})$ in
$\mathbb{P}^8$ as 
\begin{gather*}
N = \{ y_1=y_2= y_3 = y_7 + 7y_9 = y_{10} =0 \} 
\end{gather*}
(cf. (24)). Thus we get that $N(0) \cap N= \emptyset$, and the same holds in an open $N(H)$-invariant
neighbourhood of $0$ in $R$.\\
For Theorem 3.3.1, (2), (iii), it suffices to check that $\pi_0$ maps the fibre $p_{\mathbb{P}^8} (p_R^{-1} (0))$
dominantly onto $N$, which can be done by direct calculation. This concludes the proof.
\end{proof}

\begin{appendix}
\section{Collection of formulas for section 2}

We start with some remarks on how to calculate equivariant projections, and then we give explicit formulas for
the equivariant maps in section 2.\\
Let $a$, $b$ be nonnegative integers, $m:=\mathrm{min} (a,b)$, and let
$G:=\mathrm{SL}_3\,\mathbb{C}$. We denote the irreducible $G$-module whose highest weight
has numerical labels $a$, $b$ by $V(a,b)$. For $k=0,\dots , m$ we define
$V^k:=\mathrm{Sym}^{a-k}\, \mathbb{C}^3 \otimes \mathrm{Sym}^{b-k}
(\mathbb{C}^3)^{\vee}$. Let $e_1,\: e_2, \: e_3$ be the standard basis in $\mathbb{C}^3$
and $x_1,\: x_2 , \: x_3$ the dual basis in $(\mathbb{C}^3)^{\vee}$.\\
There are $G$-equivariant linear maps $\Delta^k : V^k \to V^{k+1}$ for $k=0,\dots ,
m-1$ and
$\delta^k : V^k
\to V^{k-1}$ for $k=1, \dots , m$ given by

\begin{gather}
\Delta^k:= \sum_{i=1}^3 \frac{\partial}{\partial e_i}\otimes \frac{\partial}{\partial
x_i} ,\quad \delta^k:=\sum_{i=1}^3 e_i\otimes x_i \: .
\end{gather}

(The superscript $k$ thus only serves as a means to remember the sources and targets of
the respective maps). If for some positive integers $\alpha ,\:\beta $ the $G$-module
$V^k$ contains a $G$-submodule isomorphic to $V(\alpha , \beta)$ we will denote it by
$V^k(\alpha , \beta)$ to indicate the ambient module (this is unambiguous because it is
known that all such modules occur with multiplicity one).\\
It is clear that  $\Delta^k$ is surjective and $\delta^k$ injective; one knows that
$\mathrm{ker} (\Delta^k)=V^k(a-k , b-k)$ whence

\begin{gather}
V^k=\bigoplus_{i=k}^{m} V^k(a-i , b-i) \: .
\end{gather} 
 We want to find a formula for the $G$-equivariant projection of $V^0=\mathrm{Sym}^{a}\,
\mathbb{C}^3 \otimes \mathrm{Sym}^{b} (\mathbb{C}^3)^{\vee}$ onto the subspace $V^0(a-i,
b-i)$ for $i=0, \dots , m$. We call this linear map $\pi^i_{a,b}$.\\
We remark that, by (53), one can decompose each vector $v\in V^0$ as $v=v_0+\dots +v_m$
where $v_i\in V^0(a-i, b-i)$, and this decomposition is unique. Note that

\begin{gather}
\delta^1\dots \delta^i (\mathrm{ker}\, \Delta^i)=V^0(a-i, b-i)
\end{gather}

so that

\begin{gather*}
V^0= \mathrm{ker} \,\Delta^0 \oplus \delta^1 (\mathrm{ker} \,\Delta^1) \oplus
\delta^1\delta^2 (\mathrm{ker} \,\Delta^2) \oplus\dots \oplus \delta^1\dots
\delta^i(\mathrm{ker}\, \Delta^i)\\
\oplus \dots \oplus \delta^1\dots \delta^m (V^m) \: .
\end{gather*}

Of course, $\pi^i_{a,b}(v)=v_i$. It will be convenient to put

\begin{gather}
L^i:=\delta^1\circ\delta^2 \circ \dots \circ \delta^{i}\circ\Delta^{i-1}\circ \dots \circ
\Delta^1\circ \Delta^0\: , \quad i=0, \dots , m 
\end{gather}
(whence $L^0$ is the identity) and 

\begin{gather}
U^i:=\Delta^{i-1}\circ\Delta^{i-2}\circ \dots\circ \Delta^0\circ \delta^1 \circ \dots
\circ \delta^{i-1} \circ \delta^i\: , \quad i=0, \dots , m 
\end{gather}
($U^0$ being again the identity). By Schur's lemma, we have

\begin{gather*}
U^i|_{V^i(a-i, b-i)}= c_i \cdot \mathrm{id}_{V^i(a-i,b-i)}
\end{gather*}
for some nonzero rational number $c_i\in\mathbb{Q}^{\ast}$. This is easy to calculate:
For example, since $e_1^{a-i}\otimes x_2^{b-i}\in \mathrm{ker}\, \Delta^i= V^i(a-i,
b-i)$, we have that $c_i$ is the unique number such that

\begin{gather}
U^i (e_1^{a-i}\otimes x_2^{b-i}) = c_i\cdot e_1^{a-i}\otimes x_2^{b-i}\: .
\end{gather}

We will now calculate $\pi^{m-l}_{a,b}$ for $l=0, \dots , m$ by induction on $l$; the
case $l=0$ can be dealt with as follows:\\
Write $v=v_1+\dots +v_m\in V^0$ as before. Then $v_m=\delta^1\delta^2\dots \delta^m
(u_m)$ for some $u_m\in V^m$. Now

\begin{gather*}
L^m(v)= L^m(v_m) = L^m( \delta^1\delta^2\dots \delta^m (u_m) )\\
= \delta^1\delta^2\dots \delta^m\circ U^m (u_m) = c_m v_m
\end{gather*}

so we set

\begin{gather}
\pi_{a,b}^m :=  \frac{1}{c_m} L^m \: .
\end{gather}

Now assume, by induction, that $\pi^{m-l}_{a,b} , \: \pi^{m-l+1}_{a,b}, \dots ,
\pi^{m}_{a,b}$ have already been determined. We show how to calculate
$\pi^{m-l-1}_{a,b}$.\\
Now, by (54), $v_{m-l-1}\in \delta^1\dots \delta^{m-l-1}(\mathrm{ker}\, \Delta^{m-l-1})$.
We write $v_{m-l-1}=\delta^1 \dots \delta^{m-l-1}(u_{m-l-1})$, for some $u_{m-l-1}\in
\mathrm{ker}\, \Delta^{m-l-1}= V^{m-l-1}(a-(m-l-1), b-(m-l-1))$, and using (57) we get 

\begin{gather*}
L^{m-l-1}\left( v-\sum_{i=0}^{l}\pi^{m-i}_{a,b}(v) \right) = L^{m-l-1}(v_0+v_1+\dots +
v_{m-l-1}) \\
= L^{m-l-1}(v_{m-l-1})= L^{m-l-1}(\delta^1 \dots \delta^{m-l-1}(u_{m-l-1}))\\
=\delta^1\dots \delta^{m-l-1}\circ \Delta^{m-l-2}\dots \Delta^0 \circ \delta^1\dots
\delta^{m-l-1} (u_{m-l-1}) \\
=\delta^1\dots \delta^{m-l-1}\circ U^{m-l-1} (u_{m-l-1}) = c_{m-l-1} v_{m-l-1}\: .
\end{gather*}

So we put

\begin{gather}
\pi^{m-l-1}_{a,b} := \frac{1}{c_{m-l-1}} \left( L^{m-l-1}\left( \mathrm{id}_{V^0}
- \sum_{i=0}^{l}\pi^{m-i}_{a,b} \right)\right) \: .
\end{gather}

Formulas (52), (55), (56), (57), (58), (59) contain the algorithm to compute the
$G$-equivariant linear projection

\begin{gather*}
\pi^{i}_{a,b} : V^0 \to V^0(a-i, b-i)\subset V^0
\end{gather*}

and thus to compute the associated $G$-equivariant bilinear map

\begin{gather*}
\beta^i_{a,b} : V(a, 0) \times V(0, b) \to V(a-i, b-i)
\end{gather*}

in suitable bases in source and target (remark that $V(a,
0)=\mathrm{Sym}^a\,\mathbb{C}^3$ and $V(0,b)=\mathrm{Sym}^b\, (\mathbb{C}^3)^{\vee}$).\\

In particular, we obtain for $a=2$, $b=1$ the map 
\begin{gather}
\pi^0_{2,1} \, : \,  V^0 =\mathrm{Sym}^2 \mathbb{C}^3 \otimes
 (\mathbb{C}^3)^{\veeÊ} 
\to V(2,1) \subset V^0\\
\pi^0_{2,1} = \mathrm{id}-\frac{1}{4}\delta^1\Delta^0
\: , \nonumber
\end{gather}
for $a=b=2$ the map 
\begin{gather}
\pi^0_{2,2} \, : \,  V^0 =\mathrm{Sym}^2 \mathbb{C}^3 \otimes
\mathrm{Sym}^2 (\mathbb{C}^3)^{\veeÊ} 
\to V(2,2) \subset V^0\\
\pi^0_{2,2} = \mathrm{id}-\frac{1}{5}\delta^1\Delta^0
+\frac{1}{40}\delta^1\delta^2\Delta^1\Delta^0
\: , \nonumber
\end{gather}
and for $a=b=1$ the map
\begin{gather}
\pi^0_{1,1} \, : \,  V^0 = \mathbb{C}^3 \otimes
 (\mathbb{C}^3)^{\veeÊ} 
\to V(1,1) \subset V^0\\
\pi^0_{1,1} = \mathrm{id} - \frac{1}{3} \delta^1 \Delta^0 
\: . \nonumber
\end{gather}
In the following, we will often view elements $x\in V(a,b)$ as tensors $x= (x^{i_1, \dots , i_b}_{j_1, \dots ,
j_a}) \in (\mathbb{C}^3)^{\otimes a} \otimes (\mathbb{C}^{3\vee })^{\otimes b}=: T^b_a \mathbb{C}^3$ (the
indices ranging from $1$ to $3$) which are covariant of order
$b$ and contravariant of order $a$ via the natural inclusions
\begin{gather*}
V(a,b) \subset \mathrm{Sym}^a \mathbb{C}^3 \otimes \mathrm{Sym}^b (\mathbb{C}^3)^{\vee } \subset T^b_a
\mathbb{C}^3\, 
\end{gather*}
(the first inclusion arises since $V(a,b)$ is the kernel of $\Delta^0$, the second is a tensor product of
symmetrization maps). In particular, we have the determinant tensor $\det \in T^3_0 \mathbb{C}^3$ and its inverse
$\det^{-1} \in T^0_3 \mathbb{C}^3$. In formulas involving several tensors, we will also adopt the summation
convention throughout. Finally, we define
\begin{gather}
\mathrm{can}\, : \, T^b_a \mathbb{C}^3 \to \mathrm{Sym}^a\, \mathbb{C}^3 \otimes \mathrm{Sym}^b \,
(\mathbb{C}^3)^{\vee} \, , \\
e_{j_1}\otimes \dots \otimes e_{j_a} \otimes x_{i_1} \otimes \dots \otimes x_{i_b} \mapsto e_{j_1}\cdot
 \dots
\cdot e_{j_a} \otimes x_{i_1} \cdot \dots \cdot x_{i_b} \nonumber
\end{gather}
as the canonical projection.\\
We write down the explicit formulas for the equivariant maps in section 2. The map $\Psi : V(0,4) \to V(2,2)$
(degree $3$) is given by
\begin{gather}
\Psi (f) := \pi^0_{2,2} (\mathrm{can}(g)) \, , \\
g^{i_1 i_2}_{j_1 j_2} : = f^{i_1i_2 i_3i_4} f^{i_5i_6i_7i_8} f^{i_9i_{10} i_{11} i_{12}} {\det }^{-1}_{i_3i_5i_9}
{\det }^{-1}_{i_4i_6i_{10}}
{\det }^{-1}_{j_1i_7i_{11}} {\det }^{-1}_{j_2i_8 i_{12}} \, . \nonumber 
\end{gather}
The map $\Phi : V(2,2) \times V(0,2) \to V(2,1)$ (bilinear) is given by
\begin{gather}
\Phi (g, h) := \pi^0_{2,1} (\mathrm{can} (r)) \, , \\
r^{i_1}_{j_1 j_2} := g^{i_1i_2}_{j_1i_3} h^{i_3i_4} {\det }^{-1}_{i_2i_4j_2}\, .\nonumber
\end{gather}
The map $\epsilon : V(0,4) \times V(0,2 ) \to V(2,2)$ (bilinear) is
\begin{gather}
\epsilon (f , h) := \mathrm{can} (g) , \; g^{i_1i_2}_{j_1j_2} := f^{i_3i_4i_1i_2} h^{i_5i_6} {\det
}^{-1}_{i_3j_1i_5} {\det }^{-1}_{i_4j_2i_6} \, .
\end{gather}
The map $\zeta : V(0,4) \times V(0,2) \to V(1,1)$ (homogeneous of degree $2$ in both factors) is given by
\begin{gather}
\zeta (f , h) := \pi^0_{1,1} (a) \, , \\
a^{i_1}_{j_1} := h^{i_1i_2} h^{i_3i_4} f^{i_5i_6i_7i_8} f^{i_9i_{10} i_{11}i_{12}} {\det }^{-1}_{i_5i_9j_1} {\det
}^{-1}_{i_6i_{10}i_2} {\det }^{-1}_{i_7i_{11}i_3} {\det }^{-1}_{i_8i_{12}i_4} \, . \nonumber
\end{gather}
The map $\tilde{\gamma} : V(2,2) \to V(1,1)$ (homogeneous of degree $2$) is given by
\begin{gather}
\tilde{\gamma} := \pi^0_{1,1 } (u) \, , \; u^{i_1}_{j_1} := g^{i_1i_2}_{i_3i_4} g^{i_3i_4}_{j_1i_2} \, .
\end{gather}

\section{Collection of formulas for section 3}
In section 3.1, we saw (formula (26)) that
\begin{gather}
\delta_{\lambda} = Q_1 (x , s)a_1 + Q_2 (x,s) a_2 + Q_3 (x,s) a_3 + Q_4(x,s) a_4 + Q_5(x,s) a_5 \, .
\end{gather}
We collect here the explicit values of the $Q_i(x,s)$ (recall $\lambda = (1, 6\epsilon , 1, 6)$, $\epsilon \neq
0$):
\begin{gather}
Q_1(x,s) = \hat{Q}_1 (x) +2x_7s_1 +12 x_8s_2 +2x_9s_1 +\epsilon (12s_1s_2) +2 s_0s_1 \\
+ 48 x_2 s_4 - 48 x_3 s_5 - 2 x_4 s_3 + 16 x_5s_4 -16 x_6s_5 + \epsilon (-12 s_4^2-12s_5^2) \, , \nonumber\\
Q_2(x,s) = \hat{Q}_2 (x) + 4x_8s_1 + 12 x_9s_2 + \epsilon (2s_1^2 - 6s_2^2) + 2s_0s_2 \\
-4 x_1s_3 + 16x_2s_4 + 16 x_3s_5 - 16 x_5s_4 - 16 x_6s_5 +\epsilon (-2s_3^2-4s_4^2 + 4s_5^2) \, , \nonumber\\
Q_3 (x,s) = \hat{Q}_3 (x) + 2x_4s_1 + 12 x_1s_2 + 64 x_2s_5 + 64 x_3s_4 \\
-2x_7s_3 + 2x_9s_3 + \epsilon (12 s_2s_3 - 24 s_4s_5) + 2s_0s_3 \, , \nonumber \\
Q_4 (x,s) = \hat{Q}_4 (x) + 4x_5s_1 + 12 x_2s_1 - 12 x_5s_2 + 12 x_2s_2 \\
-8x_1s_5 - 16 x_3s_3 + 8 x_8s_4 - 8x_9s_4 + \epsilon (-6s_1s_4 - 6s_2s_4 + 6s_3s_5 ) + 2s_0s_4 \, , \nonumber\\
Q_5 (x, s) = \hat{Q}_5(x) + 4x_6s_1 + 12x_3s_1 + 12 x_6s_2 - 12x_3s_2 \\
+8x_1s_4 - 16x_2s_3 - 8x_8s_5 - 8x_9s_5 + \epsilon (6s_1s_5 - 6s_2s_5 - 6 s_3s_4) + 2s_0s_5 \, , \nonumber
\end{gather}
where
\begin{gather}
\hat{Q}_1(x) = -192 x_6^2 -192 x_3x_6 +384 x_3^2 -192 x_5^2 -192 x_2x_5 +384 x_2^2 \\
- 12 x_1x_4 + 12 x_7x_8+180 x_8x_9  ,\nonumber\\
\hat{Q}_2(x)= 64 x_6^2 -192 x_3x_6 -128 x_3^2 -64 x_5^2 +192 x_2x_5 +128 x_2^2 \\
-2x_4^2 +16 x_1^2 + 2x_7^2 - 16 x_8^2 -50x_9^2 , \nonumber \\
\hat{Q}_3(x) = 96 x_5x_6 -672 x_3x_5 -672 x_2x_6 +1248 x_2x_3 \\
-12 x_1x_7 +12 x_4x_8 + 180 x_1x_9 , \nonumber \\
\hat{Q}_4(x)= 6x_4x_6 +42 x_3x_4 +84 x_1x_6 +156 x_1x_3 \\
-6x_5x_7 -42 x_2x_7 + 24 x_5x_8 - 264 x_2x_8 + 30 x_5x_9 -30 x_2x_9 , \nonumber\\
\hat{Q}_5(x) = -6x_4x_5 -42 x_2x_4 + 84 x_1x_5 + 156 x_1x_2 \\
+ 6x_6x_7 + 42 x_3x_7 + 24 x_6x_8 - 264 x_3x_8 - 20 x_6x_9 + 30 x_3x_9 \, . \nonumber  
\end{gather}
The polynomials $q_1, \dots , q_5$ defining $\tilde{Y}_{\lambda} \subset R \times \mathbb{P}^8$ (cf. Theorem
3.2.1) are:
\begin{gather}
q_1 = (-192 r_3^2 - 192 r_3 + 384) y_1y_2 + (-192 r_2^2 - 192 r_2 + 384) y_1y_3 \\
+ (-12 r_1) y_2y_3 + 12 y_7y_8 + 180 y_8y_9 + 2 y_7y_{11} + 12 y_8y_{12} \nonumber \\
+ 2y_9 y_11 + \epsilon (12 y_{11}y_{12}) + 2 y_{10} y_{11} \, , \nonumber\\
q_2 = (64 r_3^2 - 192 r_3 - 128) y_1y_2 + (- 64 r_2^2 + 192 r_2 + 128) y_1y_3 \\
 + (-2r_1^2 + 16) y_2y_3 + 2y_7^2 - 16 y_8^2 - 50 y_9^2 + 4y_8y_{11} + 12 y_9y_{12} \nonumber \\
+ \epsilon (2y_{11}^2 - 6 y_{12}^2) + 2 y_{10}y_{12} \, , \nonumber\\
q_3 = (96 r_2r_3 - 672 r_2 - 672 r_3 + 1248) y_1 \\
-12 y_7 + 12 r_1 y_8 + 180 y_9 + 2r_1 y_{11} + 12 y_{12} \, , \nonumber \\
q_4 = (6 r_1r_3 + 42 r_1 + 84 r_3 + 156) y_2 + (- 6r_2 - 42) y_7 + (24 r_2 - 264) y_8\\
+ (30 r_2 - 30) y_9 + (4 r_2+12) y_{11} + (-12r_2 + 12) y_{12} \, , \nonumber \\
q_5 = (- 6r_1r_2 - 42 r_1 + 84 r_2 + 156 ) y_3 + (6r_1 + 42) y_7 + (24 r_3 - 264) y_8 \\
+ (-30 r_3 + 30) y_9 + (4r_3 + 12) y_{11} + (12 r_3 - 12) y_{12} \, . \nonumber
\end{gather}
\end{appendix}

\end{document}